%% file: Saddle.tex
\documentclass[journal]{IEEEtran}
\include{edits}
\include{writing_pcyou}

\include{edits_yingzhu}

\ifCLASSINFOpdf
\else
\fi
\usepackage{enumitem}

\begin{document}

\title{A Unified Analysis of Saddle Flow Dynamics: Stability and Algorithm Design}

\author{Pengcheng~You,~\IEEEmembership{Member,~IEEE},~Yingzhu~Liu,~\IEEEmembership{Student~Member,~IEEE},~and~Enrique~Mallada,~\IEEEmembership{Senior~Member,~IEEE}
\thanks{}
\thanks{}
\thanks{}
\thanks{P. You and Y. Liu are with the College of Engineering, Peking University, Beijing, China (email: \{pcyou,yzliucoe\}@pku.edu.cn).}
\thanks{E. Mallada is with the Department of Electrical and Computer Engineering, Johns Hopkins University, Baltimore, MD 21218, USA (email: mallada@jhu.edu).}
\thanks{A preliminary version of this work was presented in \cite{ym2021acc}.}
}


\IEEEoverridecommandlockouts

\maketitle

\thispagestyle{plain} 
\pagestyle{plain}

\begin{abstract}
This work examines the conditions for asymptotic and exponential convergence of saddle flow dynamics of convex-concave functions. First, we propose an observability-based certificate for asymptotic convergence, directly bridging the gap between the invariant set in a LaSalle argument and the equilibrium set of saddle flows. This certificate generalizes conventional conditions for convergence, e.g., strict convexity-concavity, and leads to a novel state-augmentation method that requires minimal assumptions for asymptotic convergence. We also show that global exponential stability follows from strong convexity-strong concavity, providing a lower-bound estimate of the convergence rate. This insight also explains the convergence of proximal saddle flows for strongly convex-concave objective functions. Our results generalize to dynamics with projections on the vector field and have applications in solving constrained convex optimization via primal-dual methods. Based on these insights, we study four algorithms built upon different Lagrangian function transformations. We validate our work by applying these methods to solve a network flow optimization and a Lasso regression problem.
\end{abstract}


\IEEEpeerreviewmaketitle

\input{main}

\input{appendix}

\bibliographystyle{IEEEtran}
\bibliography{IEEEabrv,mybibfile}

\begin{IEEEbiography}[{\includegraphics[width=1in,height=1.25in,clip,keepaspectratio]{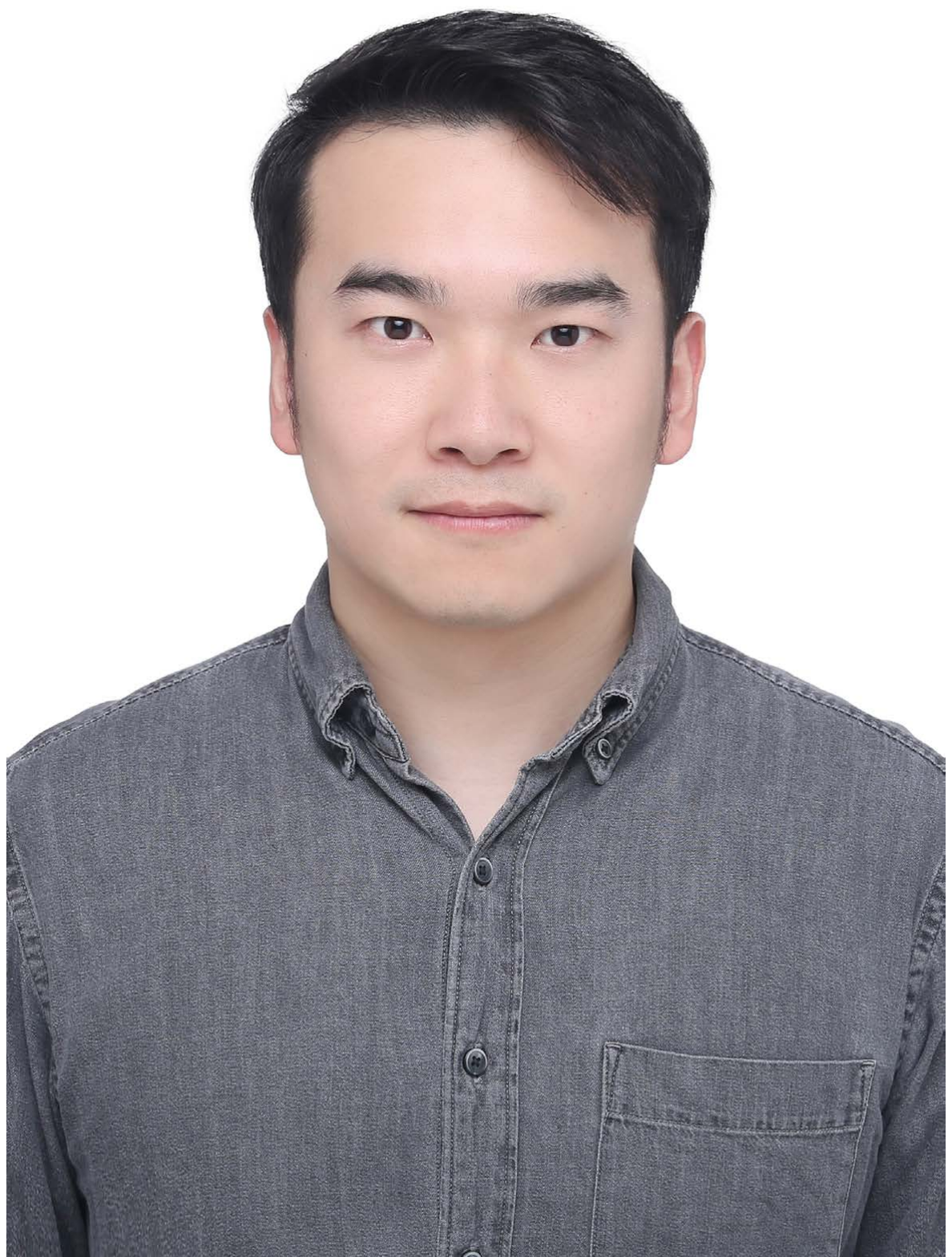}}]{Pengcheng You}
(S’14-M’18) is an Assistant Professor at the Department of Industrial Engineering and Management, Peking University. He also holds a joint appointment at the National Engineering Laboratory for Big Data Analysis and Applications, Peking University. 
Prior to joining PKU, he was a Postdoctoral Fellow at the ECE and ME Departments, Johns Hopkins University. 
He earned his Ph.D. and B.S. degrees both from Zhejiang University, China. 
During the graduate studies, he was a visiting student at Caltech and a research intern at PNNL.
His research interests include control, optimization, reinforcement learning and market mechanism, with main applications in power and energy systems.
\end{IEEEbiography}

\begin{IEEEbiography}
	[{\includegraphics[width=1in,height=1.25in, clip,keepaspectratio]{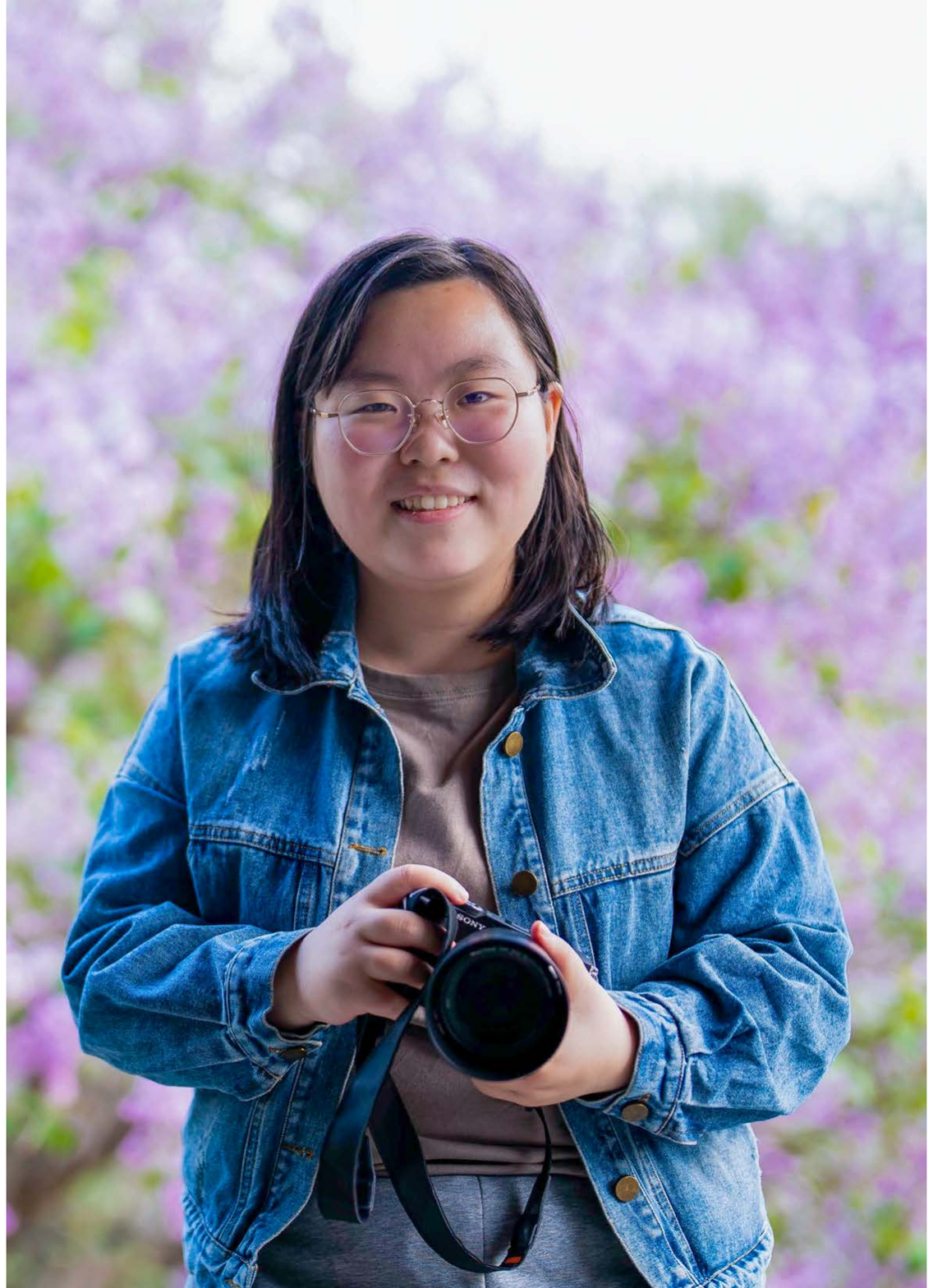}}]{Yingzhu Liu} (S'24) received her B.S. degree in mechanical engineering from 
    Peking University, Beijing, China, in 2022. She is currently working toward the Ph.D. degree in Dynamical Systems and Control with the Department of Mechanics and Engineering Science, College of Engineering, Peking University, Beijing, China.

    Her research interests include saddle flow dynamics, minimax optimization, control theory, and reinforcement learning.
\end{IEEEbiography}

\begin{IEEEbiography}[{\includegraphics[width=1in,height=1.25in,clip,keepaspectratio]{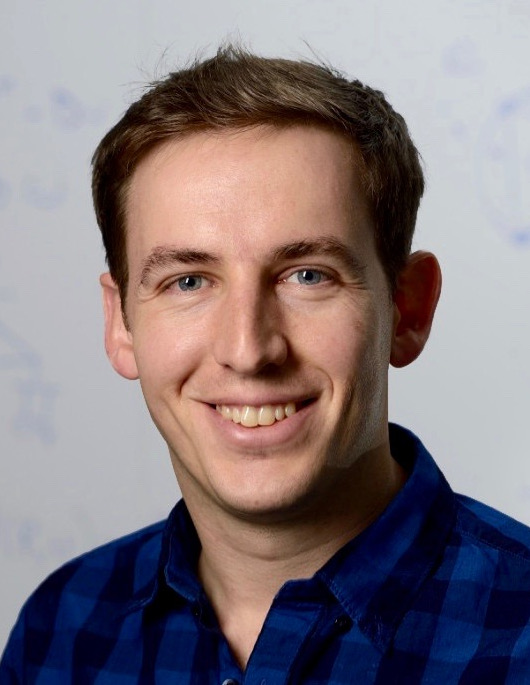}}]
{Enrique Mallada} (S'09-M'13-SM'19) is an Associate Professor of Electrical and Computer Engineering at Johns Hopkins University. He was an Assistant Professor in the same department from 2016 to 2022. Prior to joining Hopkins, he was a Post-Doctoral Fellow in the Center for the Mathematics of Information at Caltech from 2014 to 2016. He received his Ingeniero en Telecomunicaciones degree from Universidad ORT, Uruguay, in 2005 and his Ph.D. degree in Electrical and Computer Engineering with a minor in Applied Mathematics from Cornell University in 2014. 
Dr. Mallada was awarded 
the NSF CAREER award in 2018, 
the ECE Director’s PhD Thesis Research Award for his dissertation in 2014, 
the Center for the Mathematics of Information (CMI) Fellowship from Caltech in 2014, 
and the Cornell University Jacobs Fellowship in 2011. 
His research interests lie in the areas of control, dynamical systems, optimization, machine learning, with applications to engineering networks such as power systems and the Internet.
\end{IEEEbiography}

\end{document}

%% file: edits.tex
\usepackage{ifthen}
\newboolean{showcomments}
\setboolean{showcomments}{true}
\usepackage[usenames,dvipsnames]{color}

\definecolor{bleudefrance}{rgb}{0.19, 0.55, 0.91}
\definecolor{ao(english)}{rgb}{0.0, 0.5, 0.0}

\newcommand{\addcite}[0]{\ifthenelse{\boolean{showcomments}}
{\textcolor{purple}{(add cite(s)) }}{}}%

\newcommand{\you}[1]{  \ifthenelse{\boolean{showcomments}}
{{\color{blue} PCY: #1}}{}}
\newcommand{\youmargin}[1]{\ifthenelse{\boolean{showcomments}}{\marginpar{\color{bleudefrance}\tiny PCY: #1}}{}}

\newboolean{showedits}
\setboolean{showedits}{true}
\usepackage[markup=underlined]{changes}
\newcommand{\ayou}[1]{
\ifthenelse{\boolean{showedits}}
{\added[id=PCY]{#1}}
{\!#1\hspace{-4.75pt}}
}
\newcommand{\ryou}[2]{
\ifthenelse{\boolean{showedits}}
{\replaced[id=PCY]{#1}{#2}}
{\!#1\hspace{-4.75pt}}
}
\newcommand{\dyou}[1]{
\ifthenelse{\boolean{showedits}}
{\deleted[id=PCY]{#1}}
{}
}
\newcommand{\hl}[1]{\ifthenelse{\boolean{showcomments}}
{\textcolor{purple}{#1}}{}}%

\newcommand{\aem}[1]{
\ifthenelse{\boolean{showedits}}
{\added[id=EM]{#1}}
{\!#1\hspace{-4.75pt}}
}
\newcommand{\rem}[2]{
\ifthenelse{\boolean{showedits}}
{\replaced[id=EM]{#1}{#2}}
{\!#1\hspace{-4.75pt}}
}
\newcommand{\dem}[1]{
\ifthenelse{\boolean{showedits}}
{\deleted[id=EM]{#1}}
{}
}

\newcommand{\enrique}[1]{  \ifthenelse{\boolean{showcomments}}
{\todo[inline,caption={}]{Enrique: #1}}
{}
}

%% file: writing_pcyou.tex
\usepackage{epsfig, latexsym, times, bbm}
\usepackage{color}
\usepackage{colortbl}
\usepackage{array}
\usepackage{cite}

\usepackage{amsmath}
\usepackage{graphics}
\usepackage{graphicx}
\usepackage{amssymb}
\usepackage{url}
\usepackage{comment}

\usepackage{amsthm}

\usepackage{multirow}

\usepackage{bbding}
\usepackage{amsfonts}
\usepackage{cases}

\usepackage{algorithmic}
\usepackage{mathrsfs}
\usepackage{bm}
\usepackage{verbatim}
\usepackage{enumerate}

\usepackage{booktabs}
\usepackage[linesnumbered,ruled,longend]{algorithm2e}
\usepackage{textcomp}
\usepackage{subeqnarray}
\usepackage{dsfont,optidef}

\def\ba{\begin{align}}
	\def\ea{\end{align}}
\newcommand{\beq}{\begin{equation}}
	\newcommand{\eeq}{\end{equation}}
\newcommand{\bq}{\begin{eqnarray}}
	\newcommand{\eq}{\end{eqnarray}}
\newcommand{\bqn}{\begin{eqnarray*}}
	\newcommand{\eqn}{\end{eqnarray*}}
\newcommand{\bee}{\begin{enumerate}}
	\newcommand{\eee}{\end{enumerate}}
\newcommand{\bi}{\begin{itemize}}
	\newcommand{\ei}{\end{itemize}}
\newcommand{\bseq}{\begin{subequations}}
	\newcommand{\eseq}{\end{subequations}}

\newcommand{\PreserveBackslash}[1]{\let\temp=\\#1\let\\=\temp}
\newcolumntype{C}[1]{>{\PreserveBackslash\centering}p{#1}}
\newcolumntype{R}[1]{>{\PreserveBackslash\raggedleft}p{#1}}
\newcolumntype{L}[1]{>{\PreserveBackslash\raggedright}p{#1}}

\renewcommand\thepage{}

\allowdisplaybreaks

\newcommand{\R}{\mathbb{R}}
\newcommand{\norm}[1]{\|#1\|}

\newtheorem{remark}{Remark}
\newtheorem{theorem}{{Theorem}}
\newtheorem{lemma}[theorem]{{Lemma}}
\newtheorem{definition}{{Definition}}
\newtheorem{corollary}[theorem]{{Corollary}}
\newtheorem{proposition}[theorem]{{Proposition}}

\newtheorem{assumption}{{Assumption}}

%% file: edits_yingzhu.tex
\newcommand{\yingzhu}[1]{  \ifthenelse{\boolean{showcomments}}
{{\color{cyan} Yingzhu: #1}}{}}

\usepackage[labelformat=simple]{subcaption}

\graphicspath{{figures/}}

%% file: main.tex
\section{Introduction}

Studying optimization algorithms from a continuous time dynamical-systems viewpoint has become an insightful technique in the analysis of algorithms, providing alternative means to understand their stability~\cite{holding2020stability_a,holding2020stability_b}, rate of convergence~\cite{wibisono2016variational,francca2021gradient,mohammadi2020robustness}, and robustness~\cite{mohammadi2020robustness,lessard2016analysis,richert2015robust,cherukuri2017role}.
For example, in the basic case of gradient descent dynamics for unconstrained convex optimization, the objective function monotonically decreases along trajectories towards the optimum, naturally rendering a Lyapunov function~\cite{khalil2002nonlinear}. Such realization leads to multiple extensions, including finite-time convergence~\cite{cortes2006finite,garg2020fixed}, acceleration~\cite{mohammadi2020robustness,francca2021gradient}, and time-varying optimization~\cite{rahili2016distributed,fazlyab2017prediction,zheng2020implicit}, and in many cases it can be used to help understand the convergence of the discrete-time counterparts~\cite{lu2020mathprog}. 

One prominent problem in this field is the study of saddle flow dynamics, i.e., dynamics that move in the gradient descent direction for a sub-set of variables and the gradient ascent direction for the complement. Designed for locating min-max saddle points, saddle flow dynamics, or saddle flows, are particularly suited
for solving constrained optimization problems via primal-dual methods~\cite{cherukuri2016asymptotic} and finding Nash equilibria of zero-sum games via gradient play~\cite{gharesifard2013distributed}, which have led to a broad application spectrum, including power systems~\cite{zhao2014design,mzl2017tac}, communication networks~\cite{chiang2007layering,pm2009ton}, and cloud computing~\cite{goldsztajn2019proximal}.
As a result, understanding the properties that affect the behavior of these dynamics has been a major focus of study.

The seminal work~\cite{arrow1958studies} first explored the asymptotic convergence of primal-dual dynamics using first principles. 
Since then, emerging advanced analytical tools have been used to re-validate conventional conditions or unveil new, weaker ones. 
A major set of conditions are related to the convexity properties of objective functions, including \cite{arrow1958studies}.
For instance,~\cite{cherukuri2016asymptotic} revisits the strict convexity-concavity condition in the case of discontinuous vector fields, using LaSalle's invariance principle for discontinuous Caratheodory systems. Besides, weaker conditions have been discovered, such as local strong convexity-concavity~\cite{cherukuri2017role}, convexity-linearity, or strong quasiconvexity-quasiconcavity~\cite{cherukuri2017saddle}. 
Regularization serves as an alternative to circumventing the above conditions. This includes various forms of penalty terms on equality/inequality constraints to handle the Lagrangian of constrained optimization~\cite{cherukuri2017distributed,richert2015robust}, as well as the proximal methods~\cite{goldsztajn2019proximal, Goldsztajn2020proximal}. 
Despite the merit of regularization that often relaxes conditions for convergence, the extra penalty terms may introduce couplings that require additional computation and communication overheads in distributed implementation.

More recently, the research focus has shifted toward the exponential stability of saddle flows and characterizing the convergence rate.
This has been made possible through diverse conditions and techniques. For instance, the saddle flow dynamics of Lagrangian functions augmented with specific projection are proven exponentially stable in~\cite{Cortes2019Distributeda}. However, the convergence rate estimate is not provided therein.
Instead, \cite{Qu2019Exponentialb} proposes a projection-free design of dynamics for the Lagrangian of affine inequality-constrained convex programs. It is followed up by~\cite{Tang2020Semiglobal} that further accounts for convex inequality constraints with semi-global exponential stability guarantee -- the convergence rate is estimated to be dependent on initial points.
Besides, \cite{dhingra2018proximal, Ding2018Exponentially, Ding2020Global} use the Moreau envelope of Lagrangian functions as a proxy and validate the exponential stability of the resulting proximal primal-dual gradient flows using integral quadratic constraints. 
However, with this method, accurately estimating the convergence rate is difficult, thus leading to other alternatives, including contraction \cite{Nguyen2018Contraction, Cisneros-Velarde2022Contraction}, Riemannian geometric frameworks \cite{Bansode2019Exponential}, and two-timescale approaches \cite{Ozaslan2023Tighta}.
It shall be noted that most existing results are restricted to primal-dual dynamics of Lagrangian functions, and remain highly case-specific.
Thus, the conditions required for the exponential convergence of general saddle flow dynamics and any convenient methods to estimate the convergence rate are still not fully understood.

In this paper, we develop a unified analysis for saddle flow dynamics that renders general conditions for asymptotic and exponential convergence. Our results hold for, but are not limited to, Lagrangian functions derived from optimization problems and can readily be extended for scenarios where the saddle flow is projected onto a closed convex polyhedron. The novel analysis provides a unified framework to explain the convergence of several existing algorithms in the literature and offers valuable insights for developing new ones. We showcase several new algorithms that converge under weaker conditions and/or exhibit superior performance. Specifically, the paper makes the following contributions to the literature:
\begin{itemize}[leftmargin=1em, labelsep=.5em, itemindent=0em, itemsep=2pt]
    \item \emph{Asymptotic Convergence Condition:} We provide a general certificate for the asymptotic convergence of saddle flow dynamics. The proposed certificate, which is based on an observability condition, directly connects the equilibrium set with the invariant set that is derived from applying LaSalle's invariance principle and can be shown to generalize existing asymptotic convergence results.
    \item \emph{Universal Algorithm for Convex-Concave Saddle Problems:} We further leverage our condition for asymptotic convergence to design a novel saddle flow algorithm, based on state augmentation, that is able to guarantee convergence to \emph{some} saddle point, depending on the initial condition, under minimal assumptions on convexity-concavity. In particular, our algorithm is able to converge even for bilinear functions.
    \item \emph{Exponential Convergence Condition:} We show that a strongly convex-strongly concave objective function (with respective constants $\mu$ and $q$) is sufficient for saddle flows' exponential convergence and that a lower bound $c=\min\{\mu,q\}$ on the rate of convergence can be readily established. We further show how proximal regularization on a strongly convex-concave function can achieve exponential convergence by ``shifting" some of the strong convexity, making the resulting function also strongly concave. 
    \item \emph{Projected Saddle Flows:} Our analysis on both asymptotic and exponential convergence can be readily extended to the case where saddle flow dynamics are projected so as to constrain the trajectories within a closed convex polyhedron. This is achieved without requiring any strengthening of the conditions for convergence of the unconstrained case.
    \item \emph{Applications to Constrained Convex Optimization:} Finally, we bring together all our theoretical development and analyze a wide variety of primal-dual algorithms for constrained convex optimization, new and existing ones, that lead to a novel augmented primal-dual algorithm to solve linear programs, a novel bound on the convergence rate of proximal primal-dual algorithms, a novel preconditioned primal-dual algorithm that outperforms the proximal methods for strongly convex problems with affine constraints, and a novel reduced primal-dual algorithm that further exploits separable problem structures and exponentially converges with a potentially even faster rate than the preconditioned case.
\end{itemize}
The versatility of the analysis and new algorithms are validated with numerical illustrations, showcasing the solution to a network flow optimization problem using our novel primal-dual LP algorithm and the solution to a Lasso regression problem via a combination of our novel preconditioning method with proximal regularization.

\subsubsection*{Organization}
The remainder of the paper is organized as follows.
Section~\ref{sec:saddle-flows} introduces the problem formulation with basic definitions and assumptions, followed by the main results on the asymptotic convergence of saddle flow dynamics in Section~\ref{sec:asymp_convergece}. 
The exponential stability analysis of saddle flow dynamics is demonstrated in Section~\ref{sec:exponential convergence}.
We further generalize the results to the projected version in Section~\ref{sec:proj_saddle_flow_dynamics}, with
 various applications in algorithm designs for constrained convex optimization in Section~\ref{sec:applications}.
Section~\ref{sec:simulation} provides simulation validations and Section~\ref{sec:conclusion} concludes.

\subsubsection*{Contributions w.r.t. \cite{ym2021acc}}
A preliminary version of this work appeared in \cite{ym2021acc}. This paper extends those results in many ways. First, our asymptotic convergence of the proximal regularization method in Section \ref{sec:asymp_convergece} is extended to general convex-concave saddle functions. The exponential convergence analysis in Section \ref{sec:exponential convergence}, and its extension to projected saddle flows in Section \ref{ssec:exp_convergence_projected_saddle}, is entirely novel, as well as its application to (a) characterizing conditions for the exponential convergence of proximal primal-dual algorithms (Section \ref{ssec:proximal_primal_dual}), and (b) deriving two new algorithms, the preconditioned primal-dual (Section \ref{ssec:exp_convergence_regularized_pdd}) and the reduced primal-dual (Section \ref{sec:reduced_primal_dual}), with guaranteed exponential convergence. Finally, the numerical validation in Section \ref{sec:simulation} is also a new contribution.

\subsubsection*{Notation}
Let~$\mathbb{R}$ and~$\mathbb{R}_{\ge 0}$ be the sets of real and non-negative real numbers, respectively.
$I_n\in\mathbb{R}^{n\times n}$ denotes the identity matrix of size~$n$.
Given two vectors~$x,y\in\R^n$,~$x_i$ and~$y_i$ denote their~$i^{\rm th}$ entries, respectively; and~$x \le y$ holds if~$x_i \le y_i$ holds for~$\forall i$. 
Given a continuously differentiable function~$S(x,y)\in\mathcal C^1$ with~$S:\R^n\times \R^m\mapsto \R$, we use~${\partial_x}S(x,y)\in \R^{1\times n}$ and~${\partial_y}S(x,y)\in \R^{1\times m}$ to denote the partial derivatives with respect to~$x$ and~$y$, respectively, and define~$\nabla_x S(x,y):=\left[\partial_xS(x,y)\right]^T$. 
We further use $\partial^2_{xy}S(x,y):=\partial_x \left(\nabla_y S(x,y)\right)$ to denote the corresponding second-order partial derivative.

\section{Problem Formulation}\label{sec:saddle-flows}
We consider a function~$S:\mathcal{D}\mapsto\mathbb{R}$ with~$\mathcal{D}=\mathcal{X}\times\mathcal{Y}$, where both~$\mathcal{X}\subseteq\mathbb{R}^n$ and~$\mathcal{Y}\subseteq\mathbb{R}^m$ are convex sets. 
Our goal is to study different dynamic laws that seek to converge to some saddle point~$(x_\star, y_\star)$ of~$S(x,y)$.
While in general, such questions could be asked in a setting without any further restrictions, neither the existence of saddle points nor convergence towards them is easy to guarantee.
For this paper, we focus our attention on functions~$S(x,y)$ that are \emph{convex-concave}.
\begin{definition}[Convex-Concave Functions]\label{def:convex-concave}
$S(x,y)$ is convex-concave, if~$S(\cdot,y)$ is convex for~$\forall y\in \mathcal{Y}$ and~$S(x,\cdot)$ is concave for~$\forall x\in\mathcal{X}$.
$S(x,y)$ is strictly convex-concave, if $S(x,y)$ is convex-concave, and further either~$S(\cdot,y)$ is strictly convex for~$\forall y\in\mathcal{Y}$ or~$S(x,\cdot)$ is strictly concave for~$\forall x\in\mathcal{X}$.
\end{definition}
In this case, a general definition of a saddle point of $S(x,y)$ is given as follows:
\begin{definition}[Saddle Point]\label{def:saddle-point}
A point~$(x_\star,y_\star)\in\mathcal{D}$ is a saddle point of a convex-concave function~$S(x,y)$ if 
\begin{equation}\label{eq:saddle-inequality}
   S(x_\star,y)\leq S(x_\star,y_\star) \leq S(x,y_\star)\,
\end{equation}
holds for~$\forall x\in\mathcal{X}$ and~$\forall y\in\mathcal{Y}$.
\end{definition}
Due to the convexity-concavity of~$S(x,y)$, we are specifically interested in minimizing~$S(x,y)$ over~$x$ and meanwhile maximizing~$S(x,y)$ over~$y$.
Throughout this work, we will assume that~$S(x,y)$ is continuously differentiable, i.e.,~$S(x,y)\in \mathcal C^1$, as formally summarized below.

\begin{assumption}\label{ass:paper-assumption}
$S(x,y)$ is convex-concave, continuously differentiable, and there exists at least one saddle point~$(x_\star,y_\star)\in\mathcal{D}$ satisfying~\eqref{eq:saddle-inequality}.
\end{assumption}

The continuous differentiability in Assumption~\ref{ass:paper-assumption} is introduced to simplify the exposition. 
It does not significantly limit the scope of the results as one can always derive (albeit possibly increased computational complexity) a continuously differentiable surrogate of a continuous convex-concave function by means of the Moreau Envelope~\cite{parikh2014proximal}.


Given a convex-concave function~$S(x,y)$ satisfying Assumption~\ref{ass:paper-assumption}, we refer to the following dynamic law
\begin{subequations}\label{eq:saddle-flow}
\begin{align}
    \dot x &= -\,\nabla_x S(x,y)\,,\label{eq:p-saddle-flow}\\
    \dot y &= +\,\nabla_y S(x,y)\,, \label{eq:d-saddle-flow}
\end{align}
\end{subequations}
as the saddle flow dynamics of~$S(x,y)$. 
Basically the dynamic law \eqref{eq:saddle-flow} drives the system 
in gradient descent and ascent directions for~$x$ and~$y$, respectively. 

Suppose the saddle flow~\eqref{eq:saddle-flow} does converge to an equilibrium point $(x_\star,y_\star)\in\mathcal{D}$. It has to be a stationary point of $S(x,y)$, defined by
\begin{definition}[Stationary Point]\label{def:stationary-point}
A point~$(x_\star,y_\star)\in\mathcal{D}$ is a stationary point of a function~$S(x,y)$ if 
\begin{equation}\label{eq:saddle-critical}
   \left\{
   \begin{aligned}
   \nabla_x S(x_\star,y_\star) & =0 \\
   \nabla_y S(x_\star,y_\star) & =0 
   \end{aligned} \right.
\end{equation}
holds.
\end{definition}
\begin{remark}
    In the case of a convex-concave function $S(x,y)$, any stationary point must be a saddle point:
   \begin{IEEEeqnarray*}{rCl}
       \nabla_x S(x_\star,y_\star) & =0 \ \Rightarrow \ S(x_\star,y_\star)\leq S(x,y_\star), & \ \forall x\in\mathcal{X} \, ,\\
   \nabla_y S(x_\star,y_\star) & =0 \ \Rightarrow \ S(x_\star,y)\leq S(x_\star,y_\star), &\ \forall y\in\mathcal{Y} \, .
   \end{IEEEeqnarray*}
    However, the converse does not necessarily hold. 
\end{remark}

We will mainly work with the standard form \eqref{eq:saddle-flow} of saddle flow dynamics to locate a saddle point of~$S(x,y)$.
In the following Sections~\ref{sec:asymp_convergece} and \ref{sec:exponential convergence}, we first consider a special case where the feasible domain of $S(x,y)$ is full space, i.e.,~$\mathcal{D}=\mathbb{R}^n\times\mathbb{R}^m$, to derive our main results. 
In this case Assumption~\ref{ass:paper-assumption} implies that any saddle point $(x_\star,y_\star)\in\mathbb{R}^n\times \mathbb{R}^m$ is a stationary point and satisfies \eqref{eq:saddle-critical}.
Then we move on to relax the restriction and introduce a solution using a projection on the vector field to handle more general cases in Section~\ref{sec:proj_saddle_flow_dynamics}. 


\section{Asymptotic Convergence}\label{sec:asymp_convergece}

This section presents an observable certificate for a convex-concave function $S(x,y)$ that, if exists, ensures the asymptotic convergence of the saddle flow dynamics~\eqref{eq:saddle-flow} towards a saddle point.
We show that such certificates exist in the cases of two conventional sufficient conditions -- strict convexity-concavity and proximal regularization -- and are thus weaker.
We further build on this certificate to develop a state-augmentation method that entails minimal convexity-concavity requirements on~$S(x,y)$ for saddle flow dynamics to converge to a saddle point asymptotically.

\subsection{Observable Certificates}\label{ssec:general_principle}

We first define the proposed observable certificate for $S(x,y)$.
\begin{definition}[Observable Certificate]
\label{def:certificate}
A function~$h(x,y)$ with~$h:\mathbb{R}^n \times \mathbb{R}^m \mapsto \mathbb{R}_{\ge0}^2$ is an observable certificate of~$S(x,y)$, if there exists a saddle point~$(x_\star,y_\star)$ such that
\begin{equation}\label{eq:bounded-auxiliary-function}
    \begin{bmatrix}
    S(x_\star,y_\star) - S(x_\star,y)\\
    S(x,y_\star) - S(x_\star,y_\star) 
    \end{bmatrix}\ge h(x,y) \ge 0 
\end{equation}
holds and for any trajectory~$(x(t),y(t))$ of~\eqref{eq:saddle-flow} that satisfies~$h(x(t),y(t)) \equiv 0$, we have~$\dot x \equiv 0,~ \dot y \equiv 0$.
\end{definition}
We call~$h(x,y)$ an observable certificate~\cite{khalil2002nonlinear}, due to the second property of Definition~\ref{def:certificate}, which is akin to~\eqref{eq:saddle-flow} having~$h(x,y)$ as an observable output.
Our first main result is that the existence of an observable certificate suffices to guarantee the asymptotic convergence of the saddle flow dynamics \eqref{eq:saddle-flow}, as we formally state below.
\begin{assumption}~\label{ass:auxiliary-function}
$S(x,y)$ has an observable certificate~$h(x,y)$ as given by Definition~\ref{def:certificate}.
\end{assumption}
\begin{theorem}\label{th:general-principle}
Let Assumptions~\ref{ass:paper-assumption} and~\ref{ass:auxiliary-function} hold. Then the saddle flow dynamics~\eqref{eq:saddle-flow} globally asymptotically converge to some saddle point~$(x_\star,y_\star)$ of~$S(x,y)$. 
\end{theorem}
\begin{proof}
The proof follows from applying LaSalle's invariance principle~\cite{khalil2002nonlinear} to the following candidate Lyapunov function
\begin{equation}\label{eq:Vxy}
    V(x,y) = \frac{1}{2}\norm{x-x_\star}^2 + \frac{1}{2}\norm{y-y_\star}^2 \, ,
\end{equation}
where~$(x_\star,y_\star)$ is the saddle point identified in Definition~\ref{def:certificate}.
Taking the Lie derivative of~\eqref{eq:Vxy} along the trajectory~$(x(t),y(t))$ of~\eqref{eq:saddle-flow} gives
\begin{align*}
    \dot V &=(x-x_\star)^T\dot x+(y-y_\star)^T\dot y\\
    &=(x-x_\star)^T\left[-\nabla_xS(x,y)\right]+(y-y_\star)^T\left[+\nabla_yS(x,y)\right]\\
    &=(x_\star-x)^T\nabla_xS(x,y)-(y_\star-y)^T\nabla_yS(x,y)\\
    &\leq
    S(x_\star,y)-S(x,y)-
    \left(S(x,y_\star)-S(x,y)\right)\\
    &=S(x_\star,y)-S(x,y_\star)
    \\
    &=\underbrace{S(x_\star,y) -S(x_\star,y_\star)}_{\leq0} + \underbrace{S(x_\star,y_\star) - S(x,y_\star)}_{\leq0} \, ,
\end{align*}
where the second equality follows from the dynamic law~\eqref{eq:saddle-flow}, the first inequality follows from the convexity-concavity of~$S(x,y)$, and the last inequality follows from the saddle property~\eqref{eq:saddle-inequality} of~$(x_\star,y_\star)$.

Since~\eqref{eq:Vxy} is radially unbounded, all its sub-level sets are compact. From above, it follows that the trajectories of~\eqref{eq:saddle-flow} are bounded and contained in an invariant domain
\beq
D_0(x(0),y(0)):=\left\{ (x,y) \ \vert \  V(x,y)\leq V(x(0),y(0)) \right\} \, ,
\eeq
where~$(x(0),y(0))$ is any given initial point.
LaSalle's invariance principle then implies that any trajectory of~\eqref{eq:saddle-flow} should converge to the largest invariant set 
\beq\label{eq:invariant_set}
\mathbb{S} := D_0(x(0),y(0)) \cap\left\{ (x,y)  \ \vert \  \dot V(x(t),y(t)) \equiv 0   \right\} \, .
\eeq
Given Assumption~\ref{ass:auxiliary-function},~\eqref{eq:bounded-auxiliary-function} implies that~$\mathbb{S}$ is indeed a subset of 
\begin{equation}
     \mathbb{H}=\left\{ (x,y)  \ \vert \  h(x(t),y(t)) \equiv 0   \right\} \, ,
\end{equation}
which is further a subset of the equilibrium set of~\eqref{eq:saddle-flow}, denoted as 
\begin{equation}
    \mathbb{E} := \left\{ (x,y)  \ \vert \  \dot x(t), \dot y(t)  \equiv 0   \right\} \, ,
\end{equation}
i.e.,~$\mathbb{S} \subset\mathbb{H}\subset \mathbb{E}$.

Therefore, the invariant set~$\mathbb{S}$ contains only equilibrium points. If~$\mathbb{S}$ were to be composed of isolated points -- only possible when there is a unique saddle point -- this would be sufficient to prove convergence to the (unique) saddle point. However, in general, LaSalle's invariance principle only shows asymptotic convergence to the invariant set, without guaranteeing convergence to a point within it, even in the case where the set is composed of equilibrium points.

This issue is circumvented by the fact that all the equilibria within~$\mathbb{S}$ are stable. See, e.g.,~\cite[Corollary~5.2]{bhat2003nontangency}.
Alternatively, notice that~$\mathbb{S}$ is compact, and as a result any trajectory within the~$\Omega$ limit set of~\eqref{eq:saddle-flow} has a convergent sub-sequence. Let~$(\bar x,\bar y)$ be the limit point of such a sequence. Due to~$(\bar x,\bar y)\in \mathbb{S}$, it is also a saddle point. By changing~$(x_\star,y_\star)$ in the definition of~$V(x,y)$ specifically to~$(\bar x,\bar y)$, it follows that~$0\leq V(x(t),y(t))\rightarrow 0$ holds, which implies~$(x(t),y(t))\rightarrow(\bar x,\bar y)$.
\end{proof}

Checking whether Assumption~\ref{ass:auxiliary-function} holds basically requires hunting for a qualified observable certificate~$h(x,y)$ of~$S(x,y)$.
The existence and characterization of such observable certificates are still vague from Definition~\ref{def:certificate}. We next discuss how they can be identified through concrete examples. We show that the observable certificate is indeed a weaker condition underneath some of the conventional ones required for the asymptotic convergence of the saddle flow dynamics~\eqref{eq:saddle-flow}.

\subsubsection{Strict Convexity-Concavity}

The most common condition is arguably the strict convexity-concavity of~$S(x,y)$~\cite{cherukuri2016asymptotic}.
We formalize its connection with our observable certificate as below.
\begin{assumption}\label{ass:strict-cc}
$S(x,y)$ is strictly convex-concave.
\end{assumption}
\begin{proposition}\label{prop:strict_convexity-concavity}
Let Assumptions~\ref{ass:paper-assumption} and \ref{ass:strict-cc} hold. Then the function
\beq\label{eq:observable_cert_strict_cc}
h(x,y) : =     \begin{bmatrix}
    S(x_\star,y_\star) - S(x_\star,y)\\
    S(x,y_\star) - S(x_\star,y_\star) 
    \end{bmatrix}  \, ,
\eeq
with~$(x_\star,y_\star)$ being an arbitrary saddle point of~$S(x,y)$, is an observable certificate of~$S(x,y)$.
\end{proposition}
The fact that \eqref{eq:observable_cert_strict_cc} meets the second requirement 
$$
h (x(t),y(t)) \equiv 0 \ \Rightarrow  \ \dot x \equiv 0,~\dot y \equiv 0 
$$
to be an observable certificate can be readily verified using the strict convexity-concavity of $S(x,y)$. See \cite{ym2021acc} for a full analysis.

The asymptotic convergence of the saddle flow dynamics~\eqref{eq:saddle-flow} then immediately follows from Theorem~\ref{th:general-principle}.

\begin{corollary}
Let Assumptions~\ref{ass:paper-assumption} and \ref{ass:strict-cc} hold.
Then the saddle flow dynamics~\eqref{eq:saddle-flow} globally asymptotically converge to some saddle point~$(x_\star,y_\star)$ of~$S(x,y)$.
\end{corollary}

\subsubsection{Proximal Regularization}

Proximal regularization has been widely recognized as an alternative method to guarantee the asymptotic convergence of proximal primal-dual dynamics - an exemplar of saddle flows~\cite{goldsztajn2019proximal,Goldsztajn2020proximal}.
More generally, a surrogate function
\beq\label{eq:surrogate_ME}
\tilde S(u,y):= \min_{x\in\mathbb{R}^n} \left\{ S(x,y) + \frac{\rho}{2}\Vert x-u \Vert^2 \right\}
\eeq
can be defined on~$S(x,y)$, where $\rho >0$ is a constant regularization coefficient (throughout the paper).
An important fact is that the surrogate function $\tilde{S}(u,y)$ is also convex-concave and always maintains the same set of saddle points as the original function $S(x,y)$, as stated in the following lemma.
\begin{lemma}\label{lm:convex-concave-continuity-Moreau-reg}
Let Assumption~\ref{ass:paper-assumption} hold. Then $\tilde{S}(u,y)$ is convex in~$u$, concave in~$y$ and continuously differentiable with the gradient:
\begin{subequations}\label{eq:proximal_gradient}
\begin{align}
    \nabla_u \tilde{S}(u,y) & = \rho u- \rho \tilde x(u,y)  \, ,\\
    \nabla_y \tilde{S}(u,y) & = \nabla_y S(\tilde x(u,y),y) \, ,
\end{align}
\end{subequations}
where $\tilde x(u,y)$ is the unique minimizer in \eqref{eq:surrogate_ME} satisfying
\begin{equation}\label{eq:minimizer_proximal}
    \nabla_x S(\tilde x(u,y),y) + \rho(\tilde x(u,y)-u) = 0\,.
\end{equation}
Moreover, a point~$(x_\star,y_\star)$ is a saddle point of~$S(x,y)$ if and only if~$(u_\star,y_\star)$ is a saddle point of~$\tilde S(u,y)$ with~$u_\star=x_\star$.
\end{lemma}
The proof essentially follows~\cite[Theorem 2]{Goldsztajn2020proximal}, with a general convex-concave function replacing a convex-linear Lagrangian function. Lemma~\ref{lm:convex-concave-continuity-Moreau-reg} also implies $\tilde x(u_\star,y_\star) = x_\star$. 
It basically allows us to focus on the proximal saddle flow dynamics of~$S(u,y)$, i.e.,
\begin{subequations}\label{eq:saddle_flow_proximal}
\begin{align}
    \dot u & = - \nabla_u \tilde{S}(u,y)=-(\rho u-\rho \tilde x(u,y)) \, , \\
    \dot y & = + \nabla_y \tilde{S}(u,y)=\nabla_y S(\tilde x(u,y),y) \, ,
\end{align}
\end{subequations}
which suffice to locate a saddle point of the original function $S(x,y)$.
We formalize the connection of this method with our observable certificate as follows.
\begin{proposition}\label{prop:proximal-convergence}
Let Assumption~\ref{ass:paper-assumption} hold. Then the function
    \beq\label{eq:auxiliary_function_proximal} \tilde h(u,y) : =     \begin{bmatrix}
    \tilde S(u_\star,y_\star) - \tilde S(u_\star,y)\\
    \frac{\rho}{2} \Vert \tilde x(u,{y_\star}) - u\Vert^2 
    \end{bmatrix} \, ,
    \eeq
with $(u_\star,y_\star)$ being an arbitrary saddle point of $\tilde S(u,y)$, is an observable certificate of $\tilde S(u,y)$.
\end{proposition}
\begin{proof}
\eqref{eq:auxiliary_function_proximal} is a qualified observable certificate since it has the following two properties.
First, to show
$$
\begin{bmatrix}
    \tilde S(u_\star,y_\star) - \tilde S(u_\star,y)\\
    \tilde S(u,y_\star) - \tilde S(u_\star,y_\star)
\end{bmatrix} 
\ge \tilde h(u,y) \ge 0 \, ,
$$
the main task boils down to showing the following inequality:
\begin{align*}
    &\ \tilde S(u,y_\star)  - \tilde S(u_\star,y_\star)\\
    =&\ S(\tilde{x}(u,y_\star),y_\star)+\frac{\rho}{2}\|\tilde{x}(u,y_\star)-u\|^2 -S(\tilde{x}(u_\star,y_\star),y_\star)
    \\
     &\ -\frac{\rho}{2}\|\tilde{x}(u_\star,y_\star)-u_\star\|^2\\
    =&\ S(\tilde{x}(u,y_\star),y_\star) -S(x_\star,y_\star)+\frac{\rho}{2}\|\tilde{x}(u,y_\star)-u\|^2\\
    \geq&\ \frac{\rho}{2}\|\tilde{x}(u,y_\star)-u\|^2 \, ,
\end{align*}
where the second equality follows from $\tilde x(u_\star,y_\star)=x_\star=u_\star$ and the inequality follows from the convexity of $S(x,y)$ in $x$.

Second, to show
$$
\tilde h (u(t),y(t)) \equiv 0 \ \Rightarrow  \ \dot u \equiv 0,~\dot y \equiv 0 \, ,
$$
we refer readers to the proof of~\cite[Proposition 10]{Goldsztajn2020proximal}.
\end{proof}
Similarly, Theorem~\ref{th:general-principle} guarantees the asymptotic convergence of the proximal saddle flow dynamics \eqref{eq:saddle_flow_proximal}.
\begin{corollary}
Let Assumption~\ref{ass:paper-assumption} hold. Then the proximal saddle flow dynamics~\eqref{eq:saddle_flow_proximal} globally asymptotically converge to some saddle point~$(u_\star,y_\star)$ of~$\tilde S(u,y)$, with~$(x_\star =u_\star,y_\star)$ being a saddle point of~$S(x,y)$.
\end{corollary}
In fact, even the differentiability in Assumption~\ref{ass:paper-assumption} is not required since the surrogate~$\tilde S(u,y)$ can be continuously differentiable regardless.

\subsection{Augmented Saddle Flow Dynamics}\label{ssec:augmentation_regularization}

We further design a novel state-augmentation method that exploits our observable certificate and only requires Assumption~\ref{ass:paper-assumption} for augmented saddle flow dynamics to asymptotically converge to a saddle point.
The key of this method is to augment the domain of~$S(x,y)$ and introduce regularization terms that provide a convenient observable certificate without altering the positions of the original saddle points.
In particular, we propose a surrogate for~$S(x,y)$ via the following augmentation
\begin{equation}\label{eq:regularized-saddle}
    \hat{S}(x,\hat{x}, y,\hat{y}): = \frac{\rho}{2}\|x-\hat{x}\|^2 + S(x,y) - \frac{\rho}{2}\|y-\hat{y} \|^2 \, ,
\end{equation}
where~$\hat{x}\in \R^n$ and~$\hat{y} \in \R^m$ serve as two new sets of virtual variables. 
It can be readily verified that the saddle points remain invariant for~$\hat{S}(x,\hat{x},y,\hat{y})$.

\begin{lemma}\label{th:saddle-characterization}
Let Assumption~\ref{ass:paper-assumption} hold. Then a point~$(x_\star,y_\star)$ is a saddle point of~$S(x,y)$ if and only if~$(x_\star, \hat{x}_\star, y_\star,\hat{y}_\star)$ is a saddle point of~$\hat{S}(x,\hat{x},y,\hat{y})$, with
\begin{equation}
    x_\star=\hat{x}_\star ~\textrm{ and }~  y_\star=\hat{y}_\star  \, . \label{eq:reg-saddle-point-property}
\end{equation}
\end{lemma}
\begin{proof}
Recall the saddle property~\eqref{eq:saddle-inequality} of a saddle point, this theorem follows immediately from the equivalence of the following three arguments:
\begin{itemize}
    \item $\hat{S}(x_\star,\hat{x}_\star,y,\hat{y})\leq \hat{S}(x_\star,x_\star,y_\star,y_\star)\leq \hat{S}(x,\hat{x},y_\star,\hat{y}_\star)$ holds for $\forall (x,\hat x,y,\hat y)$;
    \item $S(x_\star,y)\!-\!\frac{\rho}{2}\norm{y\!-\!\hat{y}}^2\leq S(x_\star,y_\star)\leq S(x,y_\star)\!+\!\frac{\rho}{2}\norm{x\!-\!\hat{x}}^2$ holds for $\forall (x,\hat x,y,\hat y)$;
    \item $S(x_\star,y)\leq S(x_\star,y_\star)\leq S(x,y_\star)$ holds for $\forall (x,y)$.
\end{itemize}
Here the first and second arguments are equivalent due to the definition~\eqref{eq:regularized-saddle} of~$\hat{S}(x,\hat{x},y,\hat{y})$, while the second and third arguments are equivalent since the regularization terms attain zero at the minimum.
\end{proof}

Under Assumption \ref{ass:paper-assumption} for $S(x,y)$, the augmented function~$\hat{S}(x,\hat{x},y,\hat{y})$ is convex in~$(x,\hat{x})$, concave in~$(y,\hat{y})$, and continuously differentiable with at least one saddle point, by its definition~\eqref{eq:regularized-saddle} and Lemma~\ref{th:saddle-characterization}. Therefore, Assumption~\ref{ass:paper-assumption} also holds for~$\hat{S}(x,\hat{x},y,\hat{y})$.
Lemma~\ref{th:saddle-characterization} further ensures that whenever we locate a saddle point of~$\hat{S}(x,\hat{x},y,\hat{y})$, a saddle point of~$S(x,y)$ is attained simultaneously.
This motivates us to instead look at the saddle flow dynamics of~$\hat{S}(x,\hat{x},y,\hat{y})$.

Following~\eqref{eq:saddle-flow}, this augmented version of saddle flow dynamics is given by 
\begin{subequations}\label{eq:reg-saddle-flow}
\begin{align}
    \dot x &
    =-\,\nabla_xS(x, y)-\rho(x-\hat{x})
    \,,\label{eq:reg-saddle-x}\\
    \dot{\hat{x}} &
    =\rho(x-\hat{x})
    \,,\label{eq:reg-saddle-z}\\
    \dot y &
    =+\,\nabla_y S(x,y)-\rho(y-\hat{y})
    \,,
    \label{eq:reg-saddle-y}\\
    \dot{\hat{y}} &
    = \rho(y-\hat{y})\,.
    \label{eq:reg-saddle-w}
\end{align}
\end{subequations}
Although this dynamic law has twice as many state variables as its prototype~\eqref{eq:saddle-flow}, it is important to notice that, unlike proximal regularization~\cite{goldsztajn2019proximal,Goldsztajn2020proximal,parikh2014proximal} and quadratic regularizers on equality constraints~\cite{richert2015robust,cherukuri2017distributed}, the state augmentation does not introduce couplings and \eqref{eq:reg-saddle-flow} preserves any distributed structure that the system may originally have.

We now provide the key result on the asymptotic convergence of the augmented saddle flow dynamics~\eqref{eq:reg-saddle-flow}.

\begin{proposition}\label{th:regularization-asymptotic-convergence}
Let Assumption~\ref{ass:paper-assumption} hold. Then the function
\beq
\hat{h}(x,\hat{x},y,\hat{y}):=
\begin{bmatrix}
     \frac{\rho}{2}\Vert y-\hat{y}\Vert^2 \\
     \frac{\rho}{2} \Vert x-\hat{x} \Vert^2
\end{bmatrix}   \, 
\eeq
is an observable certificate of~$\hat{S}(x,\hat{x},y,\hat{y})$.
\end{proposition}
\begin{proof}
$\hat{h}(x,\hat{x},y,\hat{y})$ satisfies~\eqref{eq:bounded-auxiliary-function} as follows:
\begin{align*}
  &  \begin{bmatrix}
    \hat{S}(x_\star,\hat{x}_\star,y_\star,\hat{y}_\star) - \hat{S}(x_\star,\hat{x}_\star,y,\hat{y}) \\
    \hat{S}(x,\hat{x},y_\star,\hat{y}_\star) - \hat{S}(x_\star,\hat{x}_\star,y_\star,\hat{y}_\star) 
    \end{bmatrix}\\
   = &  \begin{bmatrix}
    \underbrace{S(x_\star,y_\star) - S(x_\star,y) }_{\ge 0}+ \frac{\rho}{2}\Vert y-\hat{y}\Vert^2 \\
    \underbrace{S(x,y_\star)    - S(x_\star,y_\star) }_{\ge 0}+ \frac{\rho}{2} \Vert x-\hat{x} \Vert^2
    \end{bmatrix}\\
   \ge &  \begin{bmatrix}
     \frac{\rho}{2}\Vert y-\hat{y}\Vert^2 \\
     \frac{\rho}{2} \Vert x-\hat{x} \Vert^2
    \end{bmatrix}\\
    \ge & \ 0 \, .
\end{align*}
Moreover,~$\hat{h}(x,\hat{x},y,\hat{y}) \equiv 0$ implies~$x(t) \equiv \hat{x}(t)$ and~$y(t)\equiv \hat{y}(t)$, which, by \eqref{eq:reg-saddle-w} and \eqref{eq:reg-saddle-z}, enforce~$\dot{ \hat{x}} \equiv 0,~\dot{\hat{y}} \equiv 0$. Therefore, $\hat x(t)$ and $\hat y(t)$, and thus $x(t)$ and $y(t)$, all remain constant and have reached an equilibrium point.
\end{proof}
Assumption~\ref{ass:auxiliary-function} thus holds for the augmented function~$\hat{S}(x,\hat{x},y,\hat{y})$ and the asymptotic convergence of the augmented saddle flow dynamics~\eqref{eq:reg-saddle-flow} follows immediately from Theorem~\ref{th:general-principle}. We summarize this result next.

\begin{corollary}\label{cor:convergence_aug_saddle}
Let Assumption~\ref{ass:paper-assumption} hold. Then the augmented saddle flow dynamics~\eqref{eq:reg-saddle-flow} globally asymptotically converge to some saddle point~$(x_\star,\hat{x}_\star,y_\star,\hat{y}_\star)$ of~$\hat{S}(x,\hat{x},y,\hat{y})$, with~$(x_\star,y_\star)$ being a saddle point of~$S(x,y)$.
\end{corollary}

Corollary~\ref{cor:convergence_aug_saddle} indicates that only the convexity-concavity of~$S(x,y)$ is required to asymptotically arrive at a saddle point through the augmented saddle flow dynamics~\eqref{eq:reg-saddle-flow}. 
This condition applies to general convex-concave functions, including bilinear functions, and is milder than many existing conditions in the literature. 
The fact that the state-augmented saddle flow \eqref{eq:reg-saddle-flow} asymptotically converges for any convex-concave saddle function $S(x,y)$ (even in cases where the saddle flow of $S(x,y)$ does not converge) should come as a surprise, given that the augmented function $\hat S(x, \hat x, y,\hat y)$ is identical to $S(x,y)$ on the hyperplane defined by $x=\hat x$ and $y=\hat y$. One explanation for this phenomenon is rooted in the dissipative theory. That is, \eqref{eq:reg-saddle-w} and \eqref{eq:reg-saddle-z} act as a dynamic feedback aimed at dissipating the ``energy" that prevents the regular saddle flow \eqref{eq:saddle-flow} from converging. We refer readers to \cite{Zheng2024Dissipativea} for further discussion on this subject.



\section{Exponential Convergence}\label{sec:exponential convergence}

This section explores the conditions under which the saddle flow dynamics~\eqref{eq:saddle-flow} are globally exponentially stable.
We establish exponential stability as a direct consequence of strong convexity-strong concavity of $S(x,y)$.
This insight can be used to explain the convergence behavior of the proximal saddle flow dynamics \eqref{eq:saddle_flow_proximal}, which generalizes existing results for the proximal gradient algorithms~\cite{Ding2020Global,Wang2021Exponential}.


\subsection{Strong Convexity-Strong Concavity}\label{ssec:storng-saddle-flow}

We first rewrite the saddle flow dynamics~\eqref{eq:saddle-flow} in a more compact form
\begin{equation}\label{eq:z-saddle-flow}
    \dot z = F(z)
\end{equation}
with~$z:=(x,y)$ and 
\begin{equation}\label{eq:z-dynamics}
    F(z) = \begin{bmatrix}
    -\nabla_x S(x,y)\\
    +\nabla_y S(x,y)
    \end{bmatrix}\, .
\end{equation}
We further assume absolute continuity 
for~$F(z)$.

\begin{assumption}\label{ass:absolute-continuity}
The gradient of~$S(x,y)$, i.e.,~$\nabla S(x,y):=[\partial_x S(x,y),\partial_y S(x,y)]^T$, is absolutely continuous \cite{Rudin1987Real}.
\end{assumption}

\begin{remark}
Assumption~\ref{ass:absolute-continuity} is slightly weaker than Lipschitz continuity - the common assumption used in the studies of global exponential stability of saddle flows~\cite{Qu2019Exponentialb,Ding2018Exponentially,Chen2019Exponential}.
\end{remark}

Assumption~\ref{ass:absolute-continuity} basically enables
\begin{equation}
    \partial_zF(z)=
    \begin{bmatrix}
    -\partial^2_{xx}S(x,y) & -\partial^2_{xy}S(x,y)\\
    \partial^2_{yx}S(x,y) & \partial^2_{yy}S(x,y)
    \end{bmatrix}\,,
\end{equation}
and
\begin{equation}\label{eq:S-block-hessians}
    \frac{1}{2}\!\!\left(\partial_zF(z) \!+\! \partial_zF(z)^T\!\right)\!\!=\!\!
    \begin{bmatrix}
    -\partial^2_{xx}S(x,y) \!\!\!&\!\!\! 0\\
    0 \!\!\!&\!\!\! \partial^2_{yy}S(x,y)
    \end{bmatrix}\, ,
\end{equation}
wherever the second-order partial derivatives are defined \cite{hubbard1999vector}.

We now show that strong convexity-strong concavity of $S(x,y)$ is conducive to the saddle flow's exponential convergence.

\begin{assumption}\label{ass:strong-convexity}
$S(x,y)$ is~$\mu$-strongly convex in~$x$ and~$q$-strongly concave in~$y$.
\end{assumption}
Note that $\mu>0$ and $q>0$ are given constants. Throughout the paper, such constants are always positive unless specified.
\begin{remark}
One consequence of Assumption~\ref{ass:strong-convexity} is that there is a unique saddle point~$(x_\star, y_\star)$. Further,  
\[
\hat S(x,y):=S(x,y)-\frac{\mu}{2}\|x-x_\star \|^2+\frac{q}{2}\|y-y_\star \|^2
\]
is convex in~$x$, concave in~$y$, and~$(x_\star,y_\star)$ is also a saddle point of~$\hat S(x,y)$.
\end{remark}

\begin{theorem}\label{th:exponential-convergence}
Let Assumptions~\ref{ass:paper-assumption}, \ref{ass:absolute-continuity} and~\ref{ass:strong-convexity} hold. Then the saddle flow dynamics~\eqref{eq:saddle-flow} are globally exponentially stable.
More precisely, given the (unique) saddle point $z_\star$ and any initial point $z(0)$ with $z=(x,y)$,
\[
\|z(t)-z_\star\|\leq \|z(0)-z_\star\|e^{-ct}
\]
holds with the rate~$$c=\min\{\mu,q\}>0 \, .$$
\end{theorem}

\begin{proof}
The proof features a reformulation of the Lie derivative of the Lyapunov function~\eqref{eq:Vxy} based on the fundamental theorem of calculus for absolutely continuous functions~\cite{Rudin2008Principles}.

We consider again the Lyapunov function
\[
V(z)=\frac{1}{2}\|z-z_\star\|^2=\frac{1}{2}\|x-x_\star\|^2+\frac{1}{2}\|y-y_\star\|^2 \, .
\]
Now taking the Lie derivative with respect to~\eqref{eq:z-saddle-flow} gives
\begin{align}
    \dot V(z)&=(z-z_\star)^TF(z)\nonumber\\
    &=\frac{1}{2}\left((z-z_\star)^TF(z)+F(z)^T(z-z_\star)\right)  \, .  \label{eq:dotVz}
\end{align}
Assumption~\ref{ass:absolute-continuity} allows us to write~$F(z)$ as 
\begin{equation}\label{eq:fund-theorem}
F(z) = \int_0^1\partial_z F(z(s))(z-z_\star)ds +\underbrace{F(z_\star)}_{=0} \, ,
\end{equation}
with~$z(s) = (z-z_\star)s+z_\star$, where we introduce a scalar $s\in\mathbb{R}$ and use the fact~$dz(s)=(z-z_\star)ds$.

Now substituting~\eqref{eq:fund-theorem} into~\eqref{eq:dotVz} gives
\begin{align}
    &\dot V(z)= (z-z_\star)^T\int_0^1   \frac{1}{2}\!\left(\partial_z F(z) \!+\! \partial_z F(z)^T\!\right)ds\;(z-z_\star)\nonumber\\   
    &= (z-z_\star)^T\int_0^1   \begin{bmatrix}
    -\partial^2_{xx}S(z(s)) \!&\! 0\\
    0 \!&\! \partial^2_{yy}S(z(s))
    \end{bmatrix}ds\;(z-z_\star) \, . \label{eq:dotVz2}
\end{align}
Note that up to this point, $\dot V(z)$ is exactly given by~\eqref{eq:dotVz2}. With Assumption~\ref{ass:strong-convexity}, \eqref{eq:dotVz2} can be relaxed as
\begin{align*}
    \dot V(z)&\leq-\mu\|x-x_\star\|^2-q\|y-y_\star\|^2\\
        &\leq -c\|z-z_\star\|^2=-2cV(z) \, .
\end{align*}
Then the Comparison Lemma allows us to derive the exponential convergence~\cite{khalil2002nonlinear}:
\begin{align*}
&V(z(t))\leq e^{-2ct}V(z(0))\\
\iff & \|z(t)-z_\star\|^2\leq e^{-2ct}\|z(0)-z_\star\|^2\ \\
\iff & \|z(t)-z_\star\|\leq e^{-ct}\|z(0)-z_\star\| \, .\quad &
\end{align*}\qedhere
\end{proof}

    The exponential convergence of the saddle flow dynamics \eqref{eq:saddle-flow} established in Theorem~\ref{th:exponential-convergence} is tightly connected with the notion of contractivity in contraction theory~\cite{Bullo2022Contraction}.
    It can be shown that, under the same condition, the vector field $F(z)$ in \eqref{eq:z-dynamics} is infinitesimally contracting with the rate $\min\{\mu,q\}$. 
Our proof however highlights that, while contraction may be sufficient, it is not necessary for exponential convergence as long as the integrated matrix in \eqref{eq:dotVz2} can be uniformly bounded, thus opening the path for future extensions on the conditions provided in this paper.

\subsection{Proximal Saddle Flow Dynamics}\label{ssec:proximal_reg}

In this subsection, we use Theorem~\ref{th:exponential-convergence} to understand how the proximal methods in general enable the exponential convergence of saddle flows.
The proximal method has been commonly used to handle non-smooth optimization problems~\cite{parikh2014proximal}. This points to the most relevant applications in primal-dual dynamics \cite{dhingra2018proximal,Goldsztajn2020proximal,Ozaslan2022Exponential}, among which \cite{dhingra2018proximal,Ozaslan2022Exponential} have shown that running primal-dual dynamics on a proximal augmented Lagrangian function is exponentially convergent. Our analysis based on Theorem~\ref{th:exponential-convergence} aims to complement these results and provide more insights into the exponential stability of proximal saddle flows.

Consider again the surrogate convex-concave function $\tilde S(u,y)$ in \eqref{eq:surrogate_ME}. It can be re-written as
\begin{equation}
    \tilde{S}(u,y)  = S(\tilde x(u,y),y)+\frac{\rho}{2}\|\tilde x(u,y)-u\|^2 \, ,
\end{equation}
using the unique minimizer $\tilde x(u,y)$ that satisfies the optimality condition~\eqref{eq:minimizer_proximal}, given~$(u,y)$.
On this basis, the proximal saddle flow dynamics are given by \eqref{eq:saddle_flow_proximal}.
We make the following assumptions on the original function $S(x,y)$ that are sufficient to guarantee the exponential convergence of the proximal saddle flow dynamics~\eqref{eq:saddle_flow_proximal} to a unique saddle point.
\begin{assumption}\label{ass:convexity_exp_assumption}
 The function~$S(x,y)$ is~$\mu$-strongly convex with~$l$-Lipschitz gradient in~$x$, i.e.,~$lI \succeq \partial^2_{xx} S(x,y) \succeq \mu I$ for $\forall (x,y)$ wherever $\partial^2_{xx} S(x,y)$ is defined. 
\end{assumption}

\begin{assumption}\label{ass:Jocabian_exp_assumption}
 The matrix~$\partial^2_{yx}S(\tilde{x}(u,y),y)$ is full row rank with~$\sigma I\succeq  \partial^2_{yx}S(\tilde{x}(u,y),y)  \partial^2_{xy}S(\tilde{x}(u,y),y)
 \succeq \kappa I$ for $\forall(u,y)$ wherever the second-order partial derivatives are defined.
\end{assumption}
\begin{remark}
Assumption~\ref{ass:Jocabian_exp_assumption} is not that restrictive in practice. For instance, consider $S(x,y)=f(x)+y^TAx$, which could be a common Lagrangian function.
Assumption~\ref{ass:Jocabian_exp_assumption} basically requires that $A$ should be full row rank with 
$\sigma I\succeq  AA^T \succeq \kappa I$, a common assumption for the exponential convergence of primal-dual dynamics \cite{Qu2019Exponentialb,Bansode2019Exponential}. Assumption~\ref{ass:Jocabian_exp_assumption} could thus be regarded as a more general version.
\end{remark}


\begin{proposition}\label{thm:proximal_saddle_exp_conv}
Let Assumptions~\ref{ass:paper-assumption},~\ref{ass:absolute-continuity},~\ref{ass:convexity_exp_assumption} and~\ref{ass:Jocabian_exp_assumption} hold. 
Then the proximal saddle flow dynamics~\eqref{eq:saddle_flow_proximal} are globally exponentially stable. More precisely, given the (unique) saddle point $w_\star$ and any initial point $w(0)$ with $w:=(u,y)$,
\[
\|w(t)-w_\star\|\leq \|w(0)-w_\star\|e^{-ct}
\]
holds with the rate $$c=\min \left \{ \frac{\mu\rho}{\mu+\rho} , \frac{\kappa}{l+\rho}  \right \}>0 \, .$$
\end{proposition}
\begin{proof}
For ease of presentation, we simply use $\tilde{x}$ to represent $\tilde{x}(u,y)$ in the proof. 
The main task is to characterize the following second-order partial derivatives of $\tilde S (u,y)$:
\begin{subequations}\label{eq:saddle_sec}
\begin{equation}
\partial^2_{uu} \tilde{S}(u,y) = \rho I - \rho \partial_u\tilde{x} \, ,
\end{equation}
\begin{equation}
\begin{aligned}
\partial^2_{yy} \tilde{S}(u,y) &= \left[\partial^2_{xy}S(\tilde{x},y)\right]^T\partial_y\tilde{x}+\partial^2_{yy}S(\tilde{x},y) \, ,
\end{aligned}
\end{equation}
\end{subequations}
which can be readily derived based on the gradient in \eqref{eq:proximal_gradient}. 
To further show the properties of the second-order partial derivatives in \eqref{eq:saddle_sec}, we utilize the optimality condition \eqref{eq:minimizer_proximal} and take partial derivatives of both sides with respect to $x$ and $y$, respectively, leading to
\begin{subequations}\label{eq:Jacobi_eq}
\begin{equation}
\partial^2_{xx}S(\tilde{x},y)\partial_u\tilde{x}+\rho \partial_u\tilde{x}-\rho I=0 \, ,
\end{equation}
\begin{equation}
\begin{aligned}
\partial^2_{xx}S(\tilde{x},y)\partial_y\tilde{x} + \left[\partial^2_{yx}S(\tilde{x},y)\right]^T + \rho \partial_y\tilde{x}=0 \, .
\end{aligned}
\end{equation}
\end{subequations}
Since $S(x,y)$ is $\mu$-strongly convex in $x$, $\partial^2_{xx}S(x,y)$ is positive definite for $\forall (x,y)$. Therefore, we can re-organize \eqref{eq:Jacobi_eq} into
\begin{subequations}\label{eq:Jacobi}
    \begin{align}
        &\partial_u\tilde{x}=\rho\left(\partial^2_{xx}S(\tilde{x},y)+\rho I\right)^{-1} \, , \\
        &\partial_y\tilde{x}=-\left(\partial^2_{xx}S(\tilde{x},y)+\rho I\right)^{-1}\left[\partial^2_{yx}S(\tilde{x},y)\right]^T \, .
    \end{align}
\end{subequations}
Combining~\eqref{eq:saddle_sec} and~\eqref{eq:Jacobi} yields
\begin{subequations}\label{eq:prox_saddle_sec}
    \begin{align}
        \partial^2_{uu}\tilde{S}(u,y)=\rho I-\rho^2(\partial^2_{xx}S(\tilde{x},y)+\rho I)^{-1}\, ,
    \end{align}
    \begin{equation}
        \begin{aligned}
            \partial^2_{yy}\tilde{S}(u,y)&=\partial^2_{yy}S(\tilde{x},y)-\left[\partial^2_{xy}S(\tilde{x},y)\right]^T\\&\cdot(\partial^2_{xx}S(\tilde{x},y)+\rho I)^{-1}\left[\partial^2_{yx}S(\tilde{x},y)\right]^T \, .
        \end{aligned}\label{eq:second-derivative-of-y}
    \end{equation}
\end{subequations}
Recall $\partial^2_{yy}S(\tilde{x},y)\preceq 0$ due to the concavity of $S(x,y)$ in $y$. Then Assumption~\ref{ass:Jocabian_exp_assumption} implies the strong convexity-strong concavity of $\tilde S(u,y)$:
\begin{subequations}
\begin{small}
\begin{equation}
\partial^2_{uu} \tilde{S}(u,y) \succeq \rho I - \frac{\rho^2 }{\mu+\rho }I \succeq \frac{\mu\rho}{\mu+\rho} I \succ 0 \, ,
\end{equation}
\end{small}
\begin{small}
\begin{equation}
\begin{aligned}
\partial^2_{yy} \tilde{S}(u,y) 
\preceq&-\partial^2_{yx}S(\tilde{x},y)(\partial^2_{xx}S(\tilde{x},y)+\rho I)^{-1} \partial^2_{xy}S(\tilde{x},y)\\
\preceq& - \frac{1}{l+\rho} \partial^2_{yx}S(\tilde{x},y)\partial^2_{xy}S(\tilde{x},y) \\
\preceq & -\frac{\kappa}{l+\rho} I \prec 0 \, ,
\end{aligned}
\end{equation}
\end{small}%
\end{subequations} %
where $\partial^2_{yx}S(\tilde{x},y) = [\partial^2_{xy}S(\tilde{x},y)]^T$ has been employed.
Then the exponential stability of the proximal saddle flow dynamics \eqref{eq:saddle_flow_proximal} follows immediately from Theorem~\ref{th:exponential-convergence}.
\end{proof}
\begin{remark}
    According to Assumption~\ref{ass:Jocabian_exp_assumption}, the second term on the right-hand side of \eqref{eq:second-derivative-of-y} is positive definite and reflects the cross-term interaction between~$x$ and~$y$ in $S(x,y)$. 
    This analysis following Theorem~\ref{th:exponential-convergence} reveals the impact of such interaction on  the convergence rate of the proximal saddle flow dynamics \eqref{eq:saddle_flow_proximal}. This insight is consistent with the results from \cite{Grimmer2023Landscape} that focuses on discrete-time proximal algorithms.
\end{remark}

Proposition~\ref{thm:proximal_saddle_exp_conv} highlights the strong convexity-strong concavity of the surrogate function $\tilde S(u,y)$ (with respective constants $\frac{\mu \rho}{\mu + \rho}$ and $\frac{\kappa}{l+\rho}$), compared with the strong convexity-concavity of the original function $S(x,y)$ (with respective constants $\mu$ and $0$). The fact of 
\begin{equation*}
    \frac{\mu \rho}{\mu + \rho} < \mu~\textrm{and}~ \frac{\kappa}{l+\rho} > 0 
\end{equation*}
unveils the rationale behind the proximal regularization that enables exponential convergence of saddle flows. Further note that a larger $\rho$ increases the strong convexity constant while decreasing the strong concavity constant. This tradeoff suggests that the convergence rate estimate - the smaller between the two constants - could be optimized with a proper choice of $\rho$.
\begin{corollary}\label{cor:proximal_saddle_optimized_rate}
The fastest convergence rate estimate is attained at 
\begin{equation}
   c_\star = \frac{2\mu \kappa}{\sqrt{(\mu l-\kappa)^2+4\mu^2\kappa} +\mu l+\kappa } < \mu 
\end{equation}
by optimizing~$\rho$ to be~$\rho_\star>0$ that satisfies
\begin{equation}\label{eq:optimal_rho}
    \frac{\mu\rho_\star }{\mu+\rho_\star} = \frac{\kappa}{l+\rho_\star} \, .
    \end{equation}
\end{corollary}

The uniqueness of~$\rho_\star>0$ is straightforward from the fact that the two sides of \eqref{eq:optimal_rho} are both strictly monotone in $\rho$ and are guaranteed to intersect. However, Corollary~\ref{cor:proximal_saddle_optimized_rate} points out a bottleneck for the proximal saddle flow's convergence rate estimate, i.e., it is always limited by the strong convexity constant $\mu$ of $S(x,y)$ in $x$.




\section{Projected Saddle Flow Dynamics} \label{sec:proj_saddle_flow_dynamics}

In this section, we generalize the convergence results in Sections~\ref{sec:asymp_convergece} and \ref{sec:exponential convergence} to account for a projection defined on the vector field of the saddle flow dynamics~\eqref{eq:saddle-flow}, which is commonly introduced to constrain solution trajectories within a feasible region.

Specifically, we look at a projected version of saddle flow dynamics of a convex-concave function~$S(x,y)$ as
\begin{align}\label{eq:general_saddle_flow_proj}
    \dot{z}=\begin{bmatrix}
        \dot{x}\\
        \dot{y}
    \end{bmatrix} =\begin{bmatrix}
        \Pi_\mathcal{X}\left[x,-\nabla_{x} S(x,y) \right]\\
        \Pi_\mathcal{Y}\left[y,+\nabla_y S(x,y) \right]
    \end{bmatrix} \, .
\end{align}
Here we consider the case where $\mathcal{X}$ and $\mathcal{Y}$ are both closed convex polyhedra feasible for $x$ and $y$, respectively. The projection is explicitly defined as follows.
\begin{definition}[Vector Field Projection \cite{pm2009ton}]
Given~$p \in\mathcal{P}\subseteq\mathbb{R}^n$ where $\mathcal{P}$ is a closed convex set and~$s\in\mathbb{R}^n$, the vector field projection~$\Pi_\mathcal{P}[p,s]$ of~$s$ at~$p$ with respect to~$\mathcal{P}$ is defined as~
\begin{equation}\label{eq:proj}
     \Pi_\mathcal{P}[p,s]:=\lim_{\delta\rightarrow 0^+}\frac{\Psi_\mathcal{P}[p+\delta s]-p}{\delta},
\end{equation}
with $\delta \in \mathbb{R}$ and $\Psi_\mathcal{P}[r]:=\arg\min_{\hat p\in\mathcal{P}}\|\hat p -r\|$ denoting the point in~$\mathcal{P}$ closest to $r\in\mathbb{R}^n$.
\end{definition}

With the projection, any trajectory of \eqref{eq:general_saddle_flow_proj} will remain in the feasible set~$\mathcal{D}=\mathcal{X} \times \mathcal{Y}$ as long as it starts with a feasible initial point. 
Meanwhile, the existence and uniqueness of the solutions of \eqref{eq:general_saddle_flow_proj} are guaranteed by \cite[Theorem 2.5]{Nagurney1996Projected} when the vector field of \eqref{eq:general_saddle_flow_proj} is Lipschitz continuous. Therefore, to work with the projected saddle flow dynamics \eqref{eq:general_saddle_flow_proj}, we enhance Assumptions~\ref{ass:paper-assumption} and \ref{ass:absolute-continuity} as follows.

\begin{assumption}
    \label{ass:Lipschitz-continuity}
The gradient of~$S(x,y)$, i.e.,~$\nabla S(x,y):=[\partial_x S(x,y),\partial_y S(x,y)]^T$, is Lipschitz continuous.
\end{assumption}

On this basis, a key enabler for generalizing the convergence results to such projected saddle flow dynamics is the following property of the projection.

\begin{lemma}\label{lm:general_projection_property}
Given any closed convex set~$\mathcal{P}\subset\mathbb{R}^n$, and~$p\in\mathcal{P}$, $p_\star \in\mathcal{P}$, $s\in\mathbb{R}^n$, the following inner product inequality always holds:
\[
\langle p_\star - p,s-\Pi_\mathcal{P}[p,s]\rangle\leq 0 \, . 
\]
\end{lemma}
\begin{proof}
    We start with a variational inequality from \cite[Chapter 0.6, Corollary 1]{Aubin1984Differential}:
$$
\langle p_\star - \Psi_{\mathcal{P}}[r] , r -\Psi_{\mathcal{P}}[r] \rangle \leq 0 \, , \quad \forall r\in\mathbb{R}^n \, ,
$$
which is intuitive, considering the projection of a point onto a convex set. 

Apply the inequality to $r = p + \delta s$, and further divide both sides by $\delta>0$, then we arrive at
$$
\left\langle p_\star - \Psi_{\mathcal{P}}[p + \delta s] ,  \frac{p + \delta s -\Psi_{\mathcal{P}}[p + \delta s]}{\delta} \right \rangle \leq 0 \, .
$$
By taking the limit $\delta \rightarrow 0^{+}$, the first term of the inner product converges to $p_\star - p$ while the second term converges to $s-\Pi_\mathcal{P}[p,s]$. This completes the proof of the lemma.
\end{proof}

Following Lemma~\ref{lm:general_projection_property}, it can be readily verified that if any trajectory of \eqref{eq:saddle_flow_proj} does converges to an equilibrium point, it is also a saddle point of $S(x,y)$.

\begin{lemma}\label{lm:saddle_proj_equilibrium}
    Any equilibrium point of the projected saddle flow dynamics \eqref{eq:saddle_flow_proj} is a saddle point $(x_\star,y_\star)\in\mathcal{D}$ of $S(x,y)$, satisfying Definition~\ref{def:saddle-point}.
\end{lemma}
\begin{proof}
    Consider an equilibrium point $(x_\star,y_\star)$ subject to 
$$
    \begin{bmatrix}
        \Pi_\mathcal{X}\left[x_\star,-\nabla_{x} S(x_\star,y_\star) \right]\\
        \Pi_\mathcal{Y}\left[y_\star,+\nabla_y S(x_\star,y_\star) \right]
    \end{bmatrix} = 0 \, .
$$
When either of the element-wise projection is inactive, it boils down to the case without the projection discussed in Section~\ref{sec:saddle-flows}. Therefore, we focus on active projections. Without loss of generality, suppose $\Pi_\mathcal{X}\left[x_\star,-\nabla_{x} S(x_\star,y_\star) \right]=0$ holds. Based on Lemma~\ref{lm:general_projection_property}, we can readily attain the saddle property from the following: 
\begin{align*}
& \langle x-x_\star,-\nabla_x S(x_\star,y_\star)-\Pi_\mathcal{X}[x_\star,-\nabla_x S(x_\star,y_\star) ]\rangle\leq 0 \, ,  \\
   \Rightarrow \ & \langle x-x_\star,-\nabla_x S(x_\star,y_\star) \rangle\leq 0 \, ,  \\
   \Rightarrow \ & S(x_\star,y_\star) - S(x,y_\star)  \leq 0 \, ,
\end{align*}
which holds for $x_\star\in \mathcal{X}$ and $\forall x\in\mathcal{X}$.
\end{proof}

Next, we formally generalize the sufficiency of observable certificates (Theorem~\ref{th:general-principle}) in Section~\ref{ssec:general_principle} and the sufficiency of strong convexity-strong concavity (Theorem~\ref{th:exponential-convergence}) in Section \ref{ssec:storng-saddle-flow} to prove the asymptotic and exponential convergence, respectively, of the projected saddle flow dynamics \eqref{eq:general_saddle_flow_proj}.

\subsection{Asymptotic Convergence of Projected Saddle Flows}

The sufficiency of observable certificates to guarantee the asymptotic convergence of the projected saddle flow dynamics~\eqref{eq:general_saddle_flow_proj} to a saddle point of~$S(x,y)$ in the feasible domain $\mathcal{D}$ is summarized as follows.
\begin{theorem}\label{th:general-principle_proj-obs}
Let Assumptions~\ref{ass:paper-assumption}, \ref{ass:auxiliary-function} and \ref{ass:Lipschitz-continuity} hold. Then starting from an arbitrary initial point $(x(0),y(0))\in\mathcal{D}$, the projected saddle flow dynamics~\eqref{eq:general_saddle_flow_proj} asymptotically converge to some saddle point~$(x_\star,y_\star)\in\mathcal{D}$ of~$S(x,y)$.
\end{theorem}

\begin{proof}
Consider the same quadratic Lyapunov function~\eqref{eq:Vxy}. Taking its Lie derivative along the trajectory~$(x(t),y(t))$ of~\eqref{eq:general_saddle_flow_proj} yields
\begin{align*}
    \dot V ={}&(x-x_\star)^T\dot x+(y-y_\star)^T\dot y\\
    ={}&(x-x_\star)^T\Pi_\mathcal{X}\left[x,-\nabla_{x} S(x,y) \right]\\
    &+(y-y_\star)^T\Pi_\mathcal{Y}\left[y,+\nabla_y S(x,y) \right]\\
    ={}&(x_\star-x)^T\nabla_xS(x,y)-(y_\star-y)^T\nabla_yS(x,y) \\
     \quad{}& \;  + \underbrace{ (x_\star-x)^T \left( - \nabla_xS(x,y) -\Pi_\mathcal{X}\left[x,-\nabla_x S(x,y)\right] \right) }_{\leq 0}\\
     \quad{}& \;  + \underbrace{ (y_\star-y)^T \left( \nabla_yS(x,y)-\Pi_\mathcal{Y}\left[y,\nabla_y S(x,y)\right] \right) }_{\leq 0}\\
     \leq{}& (x_\star-x)^T\nabla_xS(x,y)-(y_\star-y)^T\nabla_yS(x,y)\\
    \leq{}&
    S(x_\star,y)-S(x,y)-
    (S(x,y_\star)-S(x,y))\\
    ={}&S(x_\star,y)-S(x,y_\star)
    \\
    ={}&  \underbrace{S(x_\star,y) -S(x_\star,y_\star)}_{\leq0} + \underbrace{S(x_\star,y_\star) - S(x,y_\star)}_{\leq0} \, , 
\end{align*}
where the key step is to use Lemma~\ref{lm:general_projection_property} in the first inequality.

We then define the largest invariant set between the on-off switches of the projection as 
\begin{small}
\beq
\mathbb{S} := D_0(x(0),y(0)) \cap\left\{ (x,y)  \ \vert \  \dot V(x(t),y(t)) \equiv 0  , ~t\in \mathbb{R}_{\ge 0}\backslash \mathbb{T} \right\} \, ,
\eeq
\end{small}%
where $\mathbb{T}$ consists of all the time epochs when the projection switches on and off. According to the invariance principle for Caratheodory systems~\cite{bacciotti2006nonpathological}, any trajectory of \eqref{eq:general_saddle_flow_proj} converges to $\mathbb{S}$. 
Then Assumption~\ref{ass:auxiliary-function} basically still implies
$\mathbb{S} \subset \mathbb{H} \subset \mathbb{E}$, i.e., the invariant set contains only equilibrium points. Following the same discussion in the proof of Theorem~\ref{th:general-principle}, any trajectory of \eqref{eq:general_saddle_flow_proj} indeed asymptotically converges to an equilibrium point, i.e., a saddle point of $S(x,y)$.
\end{proof}


\subsection{Exponential Convergence of Projected Saddle Flows}\label{ssec:exp_convergence_projected_saddle}

We then generalize the role of a strongly convex-strongly concave objective function in the exponential convergence of the projected saddle flow dynamics~\eqref{eq:general_saddle_flow_proj}.

\begin{theorem}\label{th:general_exponential-convergence_proj}
Let Assumptions~\ref{ass:paper-assumption},~\ref{ass:strong-convexity} and~\ref{ass:Lipschitz-continuity} hold. Then the projected saddle flow dynamics~\eqref{eq:general_saddle_flow_proj} are globally exponentially stable.
More precisely, given the (unique) saddle point $z_\star$ and any initial point $z(0)\in\mathcal{D}$ with $z=(x,y)$,
\[
\|z(t)-z_\star\|\leq \|z(0)-z_\star\|e^{-ct}
\]
holds with the rate~$$c=\min\{\mu,q\}>0\, .$$
\end{theorem}

\begin{proof}
Lemma~\ref{lm:general_projection_property} allows the proof pipeline of Theorem~\ref{th:exponential-convergence} to still apply here. In particular, given $z=(x,y)$, consider again the quadratic Lyapunov function
\[
V(z)=\frac{1}{2}\|z-z_\star\|^2=\frac{1}{2}\|x-x_\star\|^2+\frac{1}{2}\|y-y_\star\|^2 \, .
\]
From the proof of Theorem~\ref{th:general-principle_proj-obs}, we already have its Lie derivative with respect to \eqref{eq:general_saddle_flow_proj} that satisfies
\begin{align*}
    \dot V(z) \leq{}& (x_\star-x)^T\nabla_xS(x,y)-(y_\star-y)^T\nabla_yS(x,y)\\
    \leq{} & (z-z_\star)^TF(z) \, .
\end{align*}
Following the analysis in the proof of Theorem~\ref{th:exponential-convergence}, we can directly obtain
\begin{small}
    \begin{equation*}
         \dot V(z)\leq-\mu\|x-x_\star\|^2-q\|y-y_\star\|^2
        \leq -c\|z-z_\star\|^2=-2cV(z) \, .
    \end{equation*}
\end{small}%
Using the Comparison Lemma defined on the Dini derivative of non-differentiable functions \cite{khalil2002nonlinear}, we can still derive the exponential convergence:
\begin{align*}
&V(z(t))\leq e^{-2ct}V(z(0))\\
\iff & \|z(t)-z_\star\|^2\leq e^{-2ct}\|z(0)-z_\star\|^2\ \\
\iff & \|z(t)-z_\star\|\leq e^{-ct}\|z(0)-z_\star\| \, .\quad &
\end{align*}\qedhere
\end{proof}

Theorems \ref{th:general-principle_proj-obs} and \ref{th:general_exponential-convergence_proj} significantly broaden the use of the proposed conditions.
In the next section we will demonstrate how they could be applied to algorithm analysis and design for solving constrained convex programs.


\section{Applications to Constrained Convex Optimization}\label{sec:applications}

We present four algorithms for constrained convex optimization. Three of them are novel and inspired by the proposed convergence results. Their asymptotic/exponential stability either requires weaker conditions and/or exhibits superior performance, i.e., faster convergence rate estimates. For an existing algorithm based on proximal primal-dual dynamics, we use our convergence results to develop a novel bound on its convergence rate.
All the algorithms are rooted in different transformations of Lagrangian functions that aim to meet our convergence conditions while preserving saddle points. 
This provides a systematic approach to algorithm design.



The Lagrangian of an inequality-constrained convex program is, in general, convex in~$x\in\mathbb{R}^n$ and linear in~$y\in\mathbb{R}^m_{\ge 0}$.  
Due to non-negative dual variables, we will, in particular, focus on the vector field projection of the form $[s]_{p}^+:= \Pi_{\mathbb{R}_{\ge 0}^{m}}[p,s]$.
In this case, this projection can be implemented element-wise, i.e., 
$$
[s_j]_{p_j}^+  =     \left\{
    \begin{aligned}
    & s_j\, ,   \qquad\qquad \quad \ \textrm{if~}p_j > 0\, , \\
    & \max\left\{s_j, 0 \right\}\, ,  \quad \textrm{otherwise}\, ,
    \end{aligned}\right.
$$
for $j=1,2,\dots, m$, given $p\in \mathbb{R}_{\ge 0}^{m} $ and $s \in \mathbb{R}^{m}$.
Therefore, in this section we mainly look at the particular projected saddle flow dynamics:
\begin{subequations}\label{eq:saddle_flow_proj}
\begin{align}
    \dot x & = -\nabla_x S(x,y)\, , \\
    \dot y & = \left[+\nabla_y S(x,y) \right]^+_y\, .
\end{align}
\end{subequations}
Along any solution trajectories, $y(t)$ will be constrained to be non-negative as long as it starts with a non-negative initial point. 

\subsection{Augmented Primal-Dual Dynamics}

Theorem~\ref{th:general-principle_proj-obs} allows the state-augmentation method in Section~\ref{ssec:augmentation_regularization} to work for the projected saddle flow dynamics \eqref{eq:saddle_flow_proj}, given the same observable certificate
\[ \hat{h}(x,\hat{x},y,\hat{y})=
\begin{bmatrix}
     \frac{\rho}{2}\Vert y-\hat{y}\Vert^2 \\
     \frac{\rho}{2} \Vert x-\hat{x} \Vert^2
\end{bmatrix}  \, ,
\]
which consists of only the regularization terms, to satisfy Assumption~\ref{ass:auxiliary-function}. 
One interesting application is a novel augmented primal-dual algorithm to solve linear programs. 

Consider the following problem: 
\begin{subequations}\label{eq:constrainedLP}
\begin{eqnarray}
\min_{x\in\mathbb{R}^n} && c^T x \\
\mathrm{s.t.} && Ax-b \le 0 \ : \ y \in\mathbb{R}^m_{\ge0}
\end{eqnarray}
\end{subequations}
where $c\in\mathbb{R}^n$, $A\in\mathbb{R}^{m\times n}$, and $b\in\mathbb{R}^m$ are given.
The problem \eqref{eq:constrainedLP} corresponds to a bi-linear Lagrangian
$$L(x,y):=c^Tx + y^T(Ax-b)\, .$$
We introduce virtual variables~$\hat{x}\in\mathbb{R}^n$,~$\hat{y}\in\mathbb{R}^m$ to define the following augmented Lagrangian:
\begin{small}
    \begin{equation*}
        \hat L(x,\hat{x},y,\hat{y}):= \frac{\rho}{2}\Vert x-\hat{x} \Vert^2 + c^Tx + y^T(Ax-b) - \frac{\rho}{2}\Vert y-\hat{y} \Vert^2\, .
    \end{equation*}
\end{small}%
From Proposition~\ref{th:regularization-asymptotic-convergence} and Theorem~\ref{th:general-principle_proj-obs}, the projected saddle flow dynamics \eqref{eq:saddle_flow_proj} of the augmented Lagrangian suggest an algorithm that globally asymptotically converges to one of its saddle points:
\begin{subequations}\label{eq:reg-saddle-flow_proj}
\begin{align}
    \dot x &
    =-c-A^Ty-\rho(x-\hat{x})
    \,,\label{eq:reg-saddle-x_proj}\\
    \dot{\hat{x}} &
    =\rho(x-\hat{x})
    \,,\label{eq:reg-saddle-z_proj}\\
    \dot y &
    = \left[Ax-b - \rho(y-\hat{y}) \right]_y^+
    \,,
    \label{eq:reg-saddle-y_proj}\\
    \dot{\hat{y}} &
    = \rho(y-\hat{y})\,.
    \label{eq:reg-saddle-w_proj}
\end{align}
\end{subequations}
The augmented primal-dual dynamics \eqref{eq:reg-saddle-flow_proj} maintain the distributed structure where each agent~$i=1,2,\dots,n$ may locally manage
\begin{subequations}\label{eq:reg-saddle-distributed}
\begin{align}
    \dot x_i &
    =-c_i-A_i^T y-\rho(x_i-\hat{x}_i)
    \,,\\
    \dot{\hat{x}}_i &
    =\rho(x_i-\hat{x}_i)
    \,,
\end{align}
and/or each dual agent~$j=1,2,\dots,m$ may locally manage
\begin{align}
    \dot y_j &
    = \left[A_j x-b_j - \rho(y_j-\hat{y}_j) \right]_{y_j}^+
    \,,
    \\
    \dot{\hat{y}}_j &
    = \rho(y_j-\hat{y}_j)\,,
\end{align}
\end{subequations}
with~$A_i$ and~$A_j$ denoting the~$i^{\text{th}}$ column and the~$j^{\text{th}}$ row of~$A$, respectively.
Lemma~\ref{th:saddle-characterization} implies that any saddle point of~$\hat L (x,\hat{x},y,\hat{y})$ corresponds to a saddle point of~$L(x,y)$, i.e., an optimal primal-dual solution to the linear program~\eqref{eq:constrainedLP}. Therefore, the augmented primal-dual dynamics \eqref{eq:reg-saddle-flow_proj} can be used as a distributed linear programming solver.

 \subsection{Proximal Primal-Dual Dynamics}\label{ssec:proximal_primal_dual}

Likewise, Theorem~\ref{th:general_exponential-convergence_proj} suggests that the proximal saddle flow dynamics in Section~\ref{ssec:proximal_reg} can be extended to handle inequality-constrained convex programs, leading to the proximal primal-dual algorithm proposed in \cite{Goldsztajn2020proximal}. However, rather than the asymptotic stability established therein, our theories enable a novel analysis of the algorithm that provides a bound on its convergence rate.

Consider the following problem:
\begin{subequations}\label{eq:convex_constrained_cvx}
\begin{eqnarray}
\min_{x\in\mathbb{R}^n} && f(x) \\
\mathrm{s.t.} && g(x) \le 0 \ : \ y \in\mathbb{R}^m_{\ge 0}
\end{eqnarray}
\end{subequations}
Here~$f : \mathbb{R}^n \mapsto \mathbb{R}$ is continuously differentiable and convex. $g:\mathbb{R}^n \mapsto \mathbb{R}^m$ is a vector-valued function, and $g_j(x)$, $j=1,2,\dots,m$, is
convex with locally Lipschitz gradient.
We define the standard Lagrangian:
$$
L(x,y):=f(x)+ y^T g(x) \, ,
$$
and its surrogate function using proximal regularization:
\begin{equation}\label{eq:reduced_Lagrangian_convex_constrianed_cvx}
\begin{aligned}
    \tilde L(u,y) : =& \min_{x\in\mathbb{R}^n} \left\{ f(x)+ y^T g(x) + \frac{\rho}{2}\Vert x-u \Vert^2 \right \}\\
    = & f(\tilde x(u,y))+  y^T g(\tilde x(u,y))+ \frac{\rho}{2}\Vert \tilde x(u,y)-u \Vert^2  \, ,
\end{aligned}
\end{equation}
where~$\tilde x(u,y)$ is the unique minimizer or the unique solution to the following first-order optimality condition:
\begin{equation}
    H(x,u,y)  = 0 \, ,
\end{equation}
with
\begin{equation}
H(x,u,y) : = \nabla f(x)  + [\partial_x g(x)]^T y + \rho(x-u)\, .
\end{equation}
Note that here $\partial_x g(x)$ is the Jacobian matrix of $g(x)$. 

The proximal primal-dual dynamics, i.e., the projected saddle flow dynamics \eqref{eq:saddle_flow_proj} of $\tilde L(u,y)$, 
\begin{subequations}\label{eq:semi-primal-dual}
    \begin{align}
        \dot u &= \rho u - \rho \tilde x(u,y) \, , \\
        \dot y &= \big [g(\tilde x(u,y))\big]_y^+ \, ,
    \end{align}
\end{subequations}
enjoy exponential convergence to a (unique) saddle point if the following assumption holds.
\begin{assumption}\label{ass:convex_constrained_primal_dual_exp_assumption}
 The function~$f(x)$ is~$\mu$-strongly convex with~$l$-Lipschitz gradient, i.e.,~$lI \succeq \nabla^2 f(x) \succeq \mu I$ wherever~$\nabla^2 f(x)$ is defined. 
 The Jacobian matrix~$\partial_x g(x)$ is full row rank with~$\sigma I\succeq \partial_x g(x)[\partial_x g(x)]^T \succeq \kappa I$, and locally Lipschitz continuous row-wise.
\end{assumption}

\begin{remark}
    $\partial_x g(x)$ being full row rank implies that the linear independence constraint qualification is satisfied, and the saddle point is unique.
\end{remark}

Given Assumption~\ref{ass:convex_constrained_primal_dual_exp_assumption}, we can derive the partial derivatives of~$\tilde x(u,y)$ as 
\begin{small}
\begin{subequations}
\beq
\begin{aligned}
    \partial_u \tilde x &= - (\partial_x H)^{-1}\partial_u H \\
    &= \rho \bigg(\nabla^2 f(x) + \rho I+ \sum_{j=1}^m y_j \nabla^2 g_j(x)\bigg)^{-1} \, ,
\end{aligned}
\eeq
\begin{equation}
\begin{aligned}
    \partial_y \tilde x &= - (\partial_x H)^{-1} \partial_y H \\
    & = - \bigg(\nabla^2 f(x) + \rho I+ \sum_{j=1}^m y_j \nabla^2 g_j(x)\bigg)^{-1} [\partial_x g(x)]^T  \, ,
\end{aligned}
\end{equation}
\end{subequations}
\end{small}%
which imply the second-order partial derivatives of~$\tilde L(u,y)$:
\begin{subequations}\label{eq:proj_proximal_second_order_derivative}
\beq
\partial^2_{uu} \tilde L = \rho I - \rho  \partial_u \tilde x    \succeq \frac{\mu\rho}{\mu+\rho} I \succ 0 \, ,
\eeq
\beq
\begin{aligned}
\partial^2_{yy} \tilde L & =  \partial_x g(x) \partial_y \tilde x \\
& \preceq -   \partial_x g(x) \bigg({l I+\rho I + \sum_{j=1}^m y_j \nabla^2 g_j(x) }\bigg)^{-1} [ \partial_x g(x)]^T \\
& \prec  0  \, .
\end{aligned}
\eeq
\end{subequations}
Here we have used the fact of~$y_j \ge 0$ and~$\nabla^2 g_j(x) \succeq 0$,~$\forall j=1,2,\dots,m$.

Recall the analysis in~\eqref{eq:dotVz2} used to establish saddle flows' exponential convergence. Given the saddle point $(u_\star,y_\star)$ of $\tilde L(u,y)$, \eqref{eq:proj_proximal_second_order_derivative} suffices to guarantee that the following quadratic Lyapunov function 
$$
    \tilde V(u,y) := \frac{1}{2} \Vert u-u_\star \Vert^2 + \frac{1}{2}\Vert y-y_\star\Vert^2 
$$
is non-increasing along any trajectories of the proximal primal-dual dynamics \eqref{eq:semi-primal-dual}. 
As a result, given an arbitrary initial point~$(u(0),y(0))$, the trajectory of~\eqref{eq:semi-primal-dual} is bounded and contained in an invariant domain
$$
    \tilde D_0(u(0),y(0)) := \left \{(u,y) \ | \  \tilde V(u,y) \le \tilde V(u(0),y(0)) \right\} \, ,
$$
which is compact and convex since~$\tilde V(u,y)$ is radially unbounded. 
Suppose
$$\Vert y \Vert \le  \zeta(u(0),y(0)) \, , $$ 
$$\nabla^2 g_j(x) \preceq \gamma(u(0),y(0)) I\, ,~\forall j =1,2, \dots, m\, , $$ hold for any point~$(u,y)\in \tilde D_0(u(0),y(0))$. Then along the trajectory starting from~$(u(0),y(0))$,
\begin{equation}
    \partial^2_{yy} \tilde L \preceq - \frac{\kappa}{l+\rho + m \zeta(u(0),y(0))  \gamma(u(0),y(0)) } I  \prec  0
\end{equation}
always holds.

Akin to Theorem~\ref{th:general_exponential-convergence_proj}, it can thus be shown that the proximal primal-dual dynamics \eqref{eq:semi-primal-dual} are semi-globally exponentially stable. In other words, given the unique saddle point $w_\star$ and any arbitrary initial point~$w(0)\in\mathbb{R}^n\times \mathbb{R}^m_{\ge 0}$ with $w=(u,y)$,
\[
\|w(t)-w_\star\|\leq \|w(0)-w_\star\|e^{-ct}
\]
holds with the rate 
\begin{equation}\label{eq:semi-EC-rate}
    c=\min \left \{\frac{\mu\rho}{\mu+\rho}, \frac{\kappa}{l+ \rho + m  \zeta(w(0)) \gamma(w(0))} \right \}>0 \, ,
\end{equation}
which is initial point-dependent.
Meanwhile, according to \cite[Theorem 5]{Goldsztajn2020proximal}, the unique saddle point of $\tilde L(u,y)$ is also the saddle point of $L(x,y)$, i.e., the unique optimal primal-dual solution to the convex program~\eqref{eq:convex_constrained_cvx}.

\subsection{Preconditioned Primal-Dual Dynamics}\label{ssec:exp_convergence_regularized_pdd}
In light of the strong convexity-strong concavity condition in Theorems~\ref{th:exponential-convergence} and \ref{th:general_exponential-convergence_proj}, we have discussed the use of the proximal methods in Sections~\ref{ssec:proximal_reg} and \ref{ssec:proximal_primal_dual} to achieve the exponential convergence of saddle flows. As observed in both cases, there is a bottleneck for the estimated convergence rate 
$$
c\le \frac{\mu \rho}{\mu +\rho} < \mu \, ,
$$
i.e., the strong convexity constant $\mu$.
We propose in this subsection a new algorithm that employs a change of variables to alter the condition number of the integrand (matrix) in \eqref{eq:dotVz2} such that an accelerated convergence rate as fast as $\mu$ can be achieved.

Consider the following problem:
\begin{subequations}\label{eq:constrained_cvx_ineq}
\begin{eqnarray}
\min_{x\in\mathbb{R}^n} && f(x) \\
\mathrm{s.t.} && Ax-b \leq0 \ : \  y \in \mathbb{R}^m_{\ge 0} 
\end{eqnarray}
\end{subequations}
where $f : \mathbb{R}^n \mapsto \mathbb{R}$ is continuously differentiable and convex, and $A\in\mathbb{R}^{m\times n}$ and $b\in\mathbb{R}^m$ are given in the affine constraints.
Define its (weighted) Lagrangian as 
\begin{equation}\label{eq:adjustedLagrangian_pddyn}
    L(x,y) : = f(x) + \eta y^T(Ax-b),
\end{equation}
where~$\eta > 0$ is a constant. 
We propose the following change of variables (with abuse of the notation $u$)
\begin{equation}\label{eq:change of var}
    u:= x+\alpha A^Ty
\end{equation}
to transform~$L(x,y)$ into 
\begin{equation}\label{eq:S(u,y)-flow}
     \tilde L(u,y) 
     := f(u-\alpha A^Ty) + \eta y^T(Au-b) - \eta \alpha  \|A^Ty\|^2 \, ,
\end{equation}
where~$\alpha >0$ is also a constant. This creates an extra quadratic term in $y$ and obviously, $\tilde L(u,y)$ is still convex-concave. In fact, 
$(x_\star,y_\star)$ is a saddle point of~$L(x,y)$ if and only if $(u_\star,y_\star)$ is a saddle point of~$\tilde L(u,y)$ with~$u_\star= x_\star + \alpha A^T y_\star$, which follows from the equivalence below:
\begin{align*}
\left\{
   \begin{aligned}
  & \nabla_x L(x_\star,y_\star)  =0  \\
  & \left[\nabla_y L(x_\star,y_\star)\right]_y^+  =0 
   \end{aligned} \right. 
 \iff
  \left\{
   \begin{aligned}
  & \nabla_u \tilde L(u_\star,y_\star)  =0 \\
  & \left[\nabla_y \tilde L(u_\star,y_\star)\right]_y^+  =0 
   \end{aligned} \right.  
\end{align*}

Since the problem \eqref{eq:constrained_cvx_ineq} is a special case of \eqref{eq:convex_constrained_cvx}, we show that given the same assumption below (Assumption~\ref{ass:convex_constrained_primal_dual_exp_assumption} for the problem \eqref{eq:constrained_cvx_ineq}), the projected saddle flow dynamics \eqref{eq:saddle_flow_proj} of $\tilde L(u,y)$,
\begin{subequations}\label{eq:transformed_pddyn}
\begin{align}
    \dot u &= - \nabla f(u-\alpha A^T y)   -  \eta A^T y \, ,    \\
    \dot y &= [-\alpha A \nabla f(u-\alpha A^T y) + \eta (A u-b) - 2 \eta\alpha AA^Ty]_y^+ \, ,
\end{align}
\end{subequations}
converge exponentially to a unique saddle point with a rate as fast as $\mu$.

\begin{assumption}\label{ass:primal_dual_exp_assumption}
 The function~$f(x)$ is~$\mu$-strongly convex with~$l$-Lipschitz gradient, i.e.,~$lI \succeq \nabla^2 f(x) \succeq \mu I$ wherever~$\nabla^2 f(x)$ is defined. The matrix~$A$ is full row rank with~$\sigma I\succeq   A A^T \succeq \kappa I$.
\end{assumption}


Given Assumption~\ref{ass:primal_dual_exp_assumption}, we can readily verify the strong convexity-strong concavity of $\tilde L(u,y)$, as along as the tunable constants $\eta$ and $\alpha$ are properly chosen subject to $2\eta> l\alpha$:
\begin{subequations}\label{eq:conditioned_primal_dual_strong}
\begin{equation}
\partial^2_{uu} \tilde L = \nabla^2 f(u-\alpha A^Ty) \succeq \mu I \succ 0 \, ,
\end{equation}
\begin{equation}
\begin{aligned}
\partial^2_{yy} \tilde L &= \alpha^2   A\nabla^2 f(u-\alpha A^Ty)A^T - 2\eta\alpha  A A^T  \\
&=  A \left( \alpha^2 \nabla^2 f(u-\alpha A^T y) - 2\eta\alpha I  \right) A^T \\
&\preceq  A ( l\alpha^2 -2\eta\alpha)IA^T \\
& \preceq -\left(2{\eta}{\alpha} - l\alpha^2\right) \kappa I \\
& \prec 0  \, .
\end{aligned}
\end{equation}\end{subequations}
Following Theorem~\ref{th:general_exponential-convergence_proj}, \eqref{eq:conditioned_primal_dual_strong} implies that the preconditioned primal-dual dynamics \eqref{eq:transformed_pddyn} are globally exponentially stable, i.e., given the unique saddle point $w_\star$ and any arbitrary initial point~$w(0)\in\mathbb{R}^n\times \mathbb{R}^m_{\ge 0}$ with $w=(u,y)$, 
\[
\|w(t)-w_\star\|\leq \|w(0)-w_\star\|e^{-ct}
\]
holds with the rate~$$c=\min \left \{\mu,\left( 2{\eta}\alpha - l\alpha^2\right) \kappa \right \}>0\, .$$

\begin{remark}
We can always pick~$\eta$ and $\alpha$ to satisfy~$2\eta > l \alpha + \frac{\mu}{\kappa\alpha}$, i.e., $\left(2{\eta}{\alpha} - l\alpha^2\right) \kappa > \mu$, such that 
$$
 c =\min \left \{\mu,( 2 \eta\alpha-  l\alpha^2 )\kappa \right \} = \mu
$$
holds. In this way, the preconditioned primal-dual dynamics \eqref{eq:transformed_pddyn} 
are guaranteed to converge with a faster rate than the one obtained for the proximal primal-dual dynamics \eqref{eq:semi-primal-dual}.
\end{remark}

Due to the linear coordinate transformation, we can also directly work with the original primal-dual variables $(x,y)$ by converting the preconditioned primal-dual dynamics \eqref{eq:transformed_pddyn} to the following:
\begin{subequations}\label{eq:regularized_pddyn}
\begin{align}
        \dot x =& -\alpha A^T[-\alpha A \left( \nabla f(x)  +  \eta A^Ty\right) + \eta (A x-b)]_y^+   \\&-\nabla f(x)-\eta A^Ty \, ,
         \\
    \dot y =& [-\alpha A \left( \nabla f(x)  +  \eta A^Ty\right) + \eta (A x-b)]_y^+ \, .
\end{align}
\end{subequations}
Essentially \eqref{eq:regularized_pddyn}, inspired by
$$\dot x : = \dot u - \alpha A^T \dot y \, ,$$
yields a trajectory $(x(t),y(t))$ that satisfies 
$$x(t) \equiv u(t) -\alpha A^T y(t) $$
for any trajectory $(u(t),y(t))$ of the preconditioned primal-dual dynamics \eqref{eq:transformed_pddyn}. 
Thus, given Assumption~\ref{ass:primal_dual_exp_assumption}, \eqref{eq:regularized_pddyn} exponentially converges to the unique saddle point $(x_\star,y_\star)$ of $L(x,y)$ with $x_\star = u_\star - \alpha A^T y_\star$. Note that $(x_\star,\eta y_\star)$ is the unique optimal primal-dual solution to the convex program \eqref{eq:constrained_cvx_ineq}.

\begin{proposition}\label{th:exp-stability-primal-dual}
Let Assumption \ref{ass:primal_dual_exp_assumption} hold.
Then the dynamics~\eqref{eq:regularized_pddyn} are globally exponentially stable, given any~$\eta>0$ and $\alpha >0$ that satisfy~$2\eta> l\alpha+\frac{\mu}{\kappa \alpha}$. More precisely, given the (unique) saddle point $z_\star$ and any initial point $z(0)\in\mathbb{R}^n \times \mathbb{R}^m_{\ge 0}$ with $z=(x,y)$,
\[
\|z(t)-z_\star\|\leq K\|z(0)-z_\star\|e^{-\mu t}
\]
holds with $K:=\max\{2,2\sigma\alpha^2+1\}$.
\end{proposition}
\begin{proof}
We start with the following inequalities:
\beq\label{eq:relaxation_1}
\begin{aligned}
&\ \| x(t) -x_\star \|^2 \\
=&\ \|( u(t) -u_\star) -\alpha A^T( y(t)-y_\star)  \|^2 \\
\le &\ 2\|  u(t) -u_\star\|^2  + 2\alpha^2 \|A^T ( y(t)-y_\star) \|^2 \\
\le &\  2\|  u(t) -u_\star \|^2 + 2\sigma \alpha^2 \| y(t)-y_\star \|^2
\end{aligned}
\eeq
and
\beq\label{eq:relaxation_2}
\begin{aligned}
& \ \|  u(t) -u_\star \|^2 \\
=&\ \| ( x(t) -x_\star ) +\alpha A^T ( y(t)-y_\star) \|^2 \\
\le & \ 2\|  x(t) -x_\star \|^2  + 2\alpha^2 \|A^T ( y(t)-y_\star) \|^2 \\
\le & \ 2\| x(t) -x_\star \|^2 + 2\sigma\alpha^2  \|y(t)-y_\star\|^2 \, .
\end{aligned}
\eeq
Recall that here $\sigma$ is the largest eigenvalue of $AA^T$ from Assumption~\ref{ass:primal_dual_exp_assumption}.
Then given $w=(u,y)$, we are able to derive
\[
\begin{aligned}
& \ \|z(t)-z_\star\|\\
 \le & \  \sqrt{ 2\| u(t)-u_\star\|^2 + (2\sigma\alpha^2 +1) \|y(t)-y_\star\|^2 } \\
 \le & \  \sqrt{K} \| w(t)-w_\star\| \\
 \le & \  \sqrt{K} \|w(0)-w_\star \| e^{-\mu t} \\
 \le & \  \sqrt{K} \sqrt{ 2\| x(0)-x_\star\|^2 + (2\sigma\alpha^2 +1) \|y(0)-y_\star\|^2 } e^{-\mu t}  \\  
 \le & \  K \|z(0)-z_\star \| e^{-\mu t}  \, ,
\end{aligned}
\]
where the first inequality uses \eqref{eq:relaxation_1}, the third inequality uses the exponential stability of the preconditioned primal-dual dynamics \eqref{eq:transformed_pddyn}, and the fourth inequality uses \eqref{eq:relaxation_2}.
\end{proof}

\begin{remark}
   Despite the strongly convex-linear Lagrangian $L(x,y)$, the coordinate transformation 
   $$
   \begin{bmatrix}
       u\\y
   \end{bmatrix}
    = \begin{bmatrix}
       I & \alpha A^T \\ 0 & I
   \end{bmatrix}
   \begin{bmatrix}
       x\\y
   \end{bmatrix}
   $$
   renders strong convexity-strong concavity of the function in the $(u,y)$ space, and thus establishes the exponential stability of the corresponding saddle flow dynamics. When the trajectories are converted back to the original $(x,y)$ space, they are still exponentially convergent with the same rate, but the distance to equilibrium is not monotonically decreasing due to the constant $K$.
\end{remark}

\subsection{Reduced Primal-Dual Dynamics}\label{sec:reduced_primal_dual}

To show the versatility of our condition for saddle flows' exponential convergence, we further propose a reduced primal-dual algorithm to solve the convex program \eqref{eq:constrained_cvx_ineq} in Section~\ref{ssec:exp_convergence_regularized_pdd} if it admits the following separable structure:  
\begin{subequations}\label{eq:sep_cvx}
\begin{eqnarray}
\min_{x\in\mathbb{R}^{n}} && f(x) := f_s(x_s) + f_c(x_c) \\
\mathrm{s.t.} && A_s x_s + A_c x_c - b \le 0 \ : \ y \in \mathbb{R}^m_{\ge 0}
\end{eqnarray}
\end{subequations}
where $f : \mathbb{R}^n \mapsto \mathbb{R}$ is additively separable into $f_s:\mathbb{R}^{n_s} \mapsto \mathbb{R}$ and $f_c:\mathbb{R}^{n_c} \mapsto \mathbb{R}$. 
Accordingly, $x_s\in\mathbb{R}^{n_s}$ and~$x_c\in\mathbb{R}^{n_c}$ are the two separated subsets of variables with~$x=[x_s^T, x_c^T]^T \in\mathbb{R}^n$ and~$n=n_s+n_c$. 
Both~$f_s(x_s)$ and $f_c(x_c)$ are assumed to be continuously differentiable and convex. $A\in\mathbb{R}^{m\times n}$ and $b\in\mathbb{R}^m$ are given. $A_s\in\mathbb{R}^{m\times n_s}$ and~$A_c\in\mathbb{R}^{m\times n_c}$ are the submatrices of~$A=[A_s,A_c]$ that consist of the columns corresponding to~$x_s$ and~$x_c$, respectively. 

We make the following assumption to proceed.
\begin{assumption}\label{ass:reduced_primal_dual_dynamics}
  The function~$f_s(x)$ is~$\mu_s$-strongly convex with~$l_s$-Lipschitz gradient, i.e.,~$l_sI \succeq \nabla^2 f_s(x) \succeq \mu_sI$ wherever~$\nabla^2 f_s(x)$ is defined. Similarly,~$l_cI \succeq \nabla^2 f_c(x) \succeq \mu_cI$ holds wherever $\nabla^2 f_c(x)$ is defined. 
  The matrix~$A_s$ is full row rank with~$\sigma_s I\succeq A_sA_s^T\succeq \kappa_s I$. 
\end{assumption}
\noindent
Instead of introducing regularization, we use the standard Lagrangian
\begin{equation}
    L(x,y) := f_s(x_s) + f_c(x_c) + y^T( A_s x_s + A_c x_c - b )\, ,
\end{equation}
but minimize it over~$x_s$ to attain a reduced Lagrangian 
\begin{equation}
\begin{aligned}
    \bar L (x_c,y)  := & \min_{x_s\in\mathbb{R}^{n_s}} L(x,y) \\ 
     = & f_s(\bar x_{s}(y)) + f_c(x_{c}) + y^T( A_s \bar x_{s}(y) + A_c x_c - b )
\end{aligned}
\end{equation}
with~$\bar x_{s}(y)$ being the unique minimizer given~$y$ such that 
\begin{equation}\label{eq:partial_minimizer_Lagrangian}
    \nabla f_s (\bar x_{s}(y))  + A_s^T y=0
\end{equation}
holds.

The reduced Lagrangian $\bar L(x_c,y)$ can be exploited due to the properties below.
\begin{lemma}\label{lm:convex_concave_reduced_Lagrangian}
Let Assumption~\ref{ass:reduced_primal_dual_dynamics} hold. Then $\bar L(x_c,y)$ is strongly convex-strongly concave with the Lipschitz gradient:
\begin{subequations}
\begin{align}
    \nabla_{x_c} \bar L(x_c,y) & =  \nabla f_c(x_c) + A_c^T y \, ,\\
    \nabla_y \bar L(x_c,y) & = A_s \bar x_{s}(y) +A_c x_c -b \, .
\end{align}
\end{subequations}
Moreover, a point~$(x_\star,y_\star)$ is a saddle point of~$L(x,y)$ if and only if~$(x_{c\star},y_\star)$ is a saddle point of~$\bar L(x_c,y)$ with~$ [\bar x_{s}^T(y_\star), x^T_{c\star}]^T = x_\star$.
\end{lemma}
\begin{proof}
It follows from~\eqref{eq:partial_minimizer_Lagrangian} and the strict convexity of~$f_s$ that 
\begin{equation}
    \bar x_{s}(y) = ( \nabla f_s )^{-1} (-A_s^T y) 
\end{equation}
is Lipschitz. Denote the Jacobian matrix of $( \nabla f_s )^{-1}$ as $\textbf{J}_{s}$ wherever it is defined. 
Taking the derivative with respect to $y$ on both sides of \eqref{eq:partial_minimizer_Lagrangian} leads to the following:
\begin{equation}\label{eq:psd_Jacobian}
    ( - \nabla^2 f_s (\bar x_s(y)) \textbf{J}_s + I) A_s^T \equiv 0 
 \  \iff  \  \nabla^2 f_s (\bar x_s(y)) \textbf{J}_s  \equiv I  
\end{equation}

Then the gradient of~$\bar L(x_c,y)$ can be obtained using the chain rule:
\begin{subequations}
\beq
        \nabla_{x_c} \bar L(x_c,y)  =  \nabla f_c(x_c) + A_c^T y \, , \\
\eeq
\beq
\begin{aligned}
    \nabla_y \bar L(x_c,y) &  = - A_s \textbf{J}^T_{s}\nabla f_s(\bar x_{s}(y))  - A_s\textbf{J}^T_{s} A_s^T y  \\
    & \ \ \, \  + A_s \bar x_{s}(y) +A_c x_c -b \\
    & = A_s \bar x_{s}(y) +A_c x_c -b \, .
\end{aligned}
\eeq
\end{subequations}
This allows us to further compute the second-order partial derivatives as 
\begin{subequations}\label{eq:second-order_deriv_reduced}
\beq
        \partial^2_{x_c x_c} \bar L(x_c,y)  =  \nabla^2 f_c(x_c)  \succeq \mu_c  \, ,
\eeq
\beq
    \partial^2_{yy} \bar L(x_c,y)  =  -A_s\textbf{J}^T_{s} A^T_s \preceq -\frac{1}{l_s}A_s A^T_s\preceq -\frac{\kappa_s}{l_s} \, ,
\eeq
\end{subequations}
where $\textbf{J}^T_{s} \succeq \frac{1}{l_s}$ follows from \eqref{eq:psd_Jacobian} with $\nabla^2_{x_s} f_s(x_s)  \preceq l_s$. 

Given the strong convexity-strong concavity of $\bar L(x_c,y)$ above, the correspondence between saddle points of $L(x,y)$ and $\bar L(x_c,y)$ is straightforward from
\begin{align*}
\left\{
   \begin{aligned}
   & \nabla_x L(x_\star,y_\star)  =0 \\&
   \left[\nabla_y L(x_\star,y_\star)\right]^+_y  =0 
   \end{aligned} \right. 
 \iff
  \left\{
   \begin{aligned}
   & \nabla_{x_c} \bar L (x_{c\star},y_\star)  = 0 \\
    & \nabla_y \bar L \left[(x_{c\star},y_\star)\right]_y^+  = 0 
   \end{aligned} \right.  
\end{align*}
\end{proof}

Lemma~\ref{lm:convex_concave_reduced_Lagrangian} allows us to focus on the projected saddle flow dynamics \eqref{eq:saddle_flow_proj} of $\bar L(x_c,y)$:
\begin{subequations}\label{eq:reduced_primal_dual_dynamics}
\begin{align}
    \dot x_c & = - \nabla_{x_c} \bar L(x_c,y)\, , \\
    \dot y & = \left[+ \nabla_y \bar L(x_c,y)\right]_y^+\, ,
\end{align}
\end{subequations}
since such reduced primal-dual dynamics \eqref{eq:reduced_primal_dual_dynamics} converge exponentially to its unique saddle point according to Theorem~\ref{th:general_exponential-convergence_proj}.
In other words, given the unique saddle point $\bar z_\star$ and any initial point $\bar z(0)\in \mathbb{R}^{n_c}\times \mathbb{R}^m_{\ge 0}$ with $\bar z:=(x_c,y)$, 
\[
\|\bar z(t)-\bar z_\star\|\leq \|\bar z(0)-\bar z_\star\|e^{-ct}
\]
holds with the rate~$$c=\min \left \{  \mu_c , \frac{\kappa_s}{l_s} \right \}>0\, .$$
\begin{remark}
The key enabler for the reduced primal-dual dynamics \eqref{eq:reduced_primal_dual_dynamics} is the separability of~$x$ into~$x_s$ and~$x_c$ such that all the necessary assumptions hold. It is possible that multiple ways of separation exist, then the convergence rate could be further optimized based on how the variables are split.
In this case, comparing the reduced primal-dual dynamics \eqref{eq:reduced_primal_dual_dynamics} with the preconditioned primal-dual dynamics  \eqref{eq:transformed_pddyn} might also suggest a better option between the two since in general neither of them dominates the other in terms of the convergence rate. 
\end{remark}

\section{Simulation Results}\label{sec:simulation}
\begin{figure*}[htbp]
    \centering

    \begin{subfigure}[t]{0.25\textwidth}
        \centering
        \includegraphics[width=\linewidth]{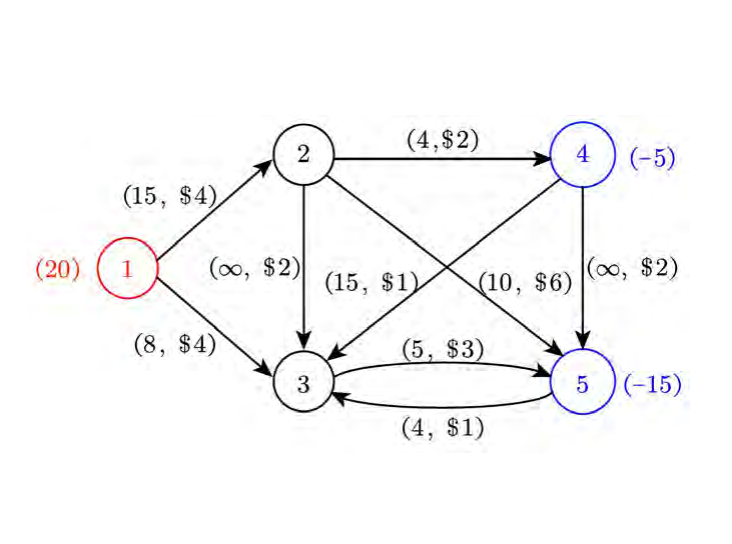}
        \caption{}
        \label{fig: 1a}
    \end{subfigure}%
    \hfill
    \begin{subfigure}[t]{0.25\textwidth}
        \centering
        \includegraphics[width=\linewidth]{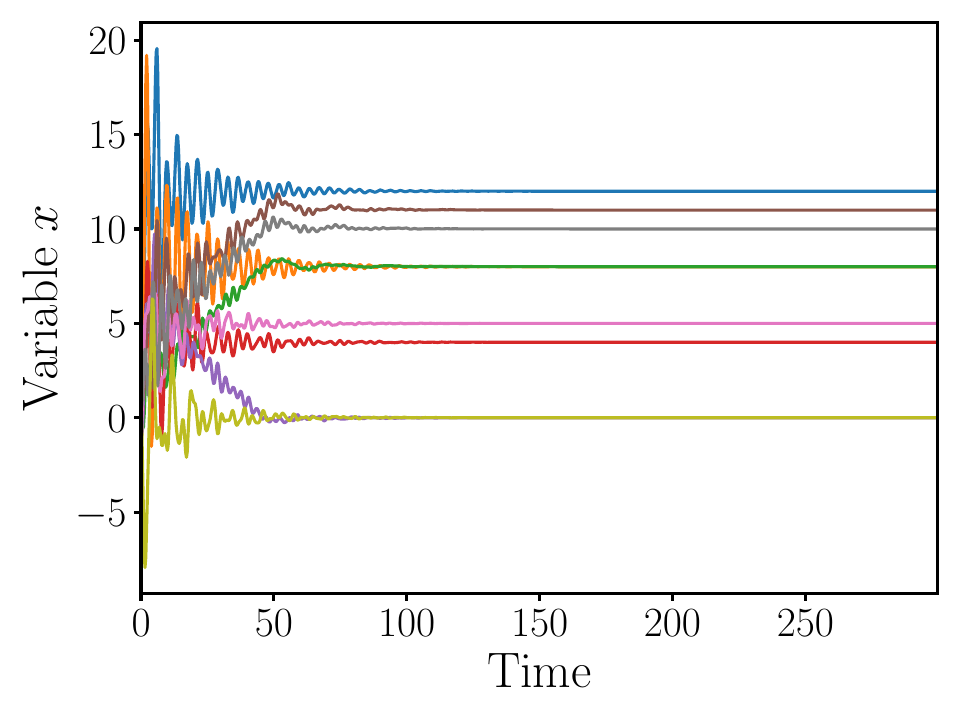}
        \caption{}
        \label{fig: 1b}
    \end{subfigure}%
    \hfill
    \begin{subfigure}[t]{0.25\textwidth}
        \centering
        \includegraphics[width=\textwidth]{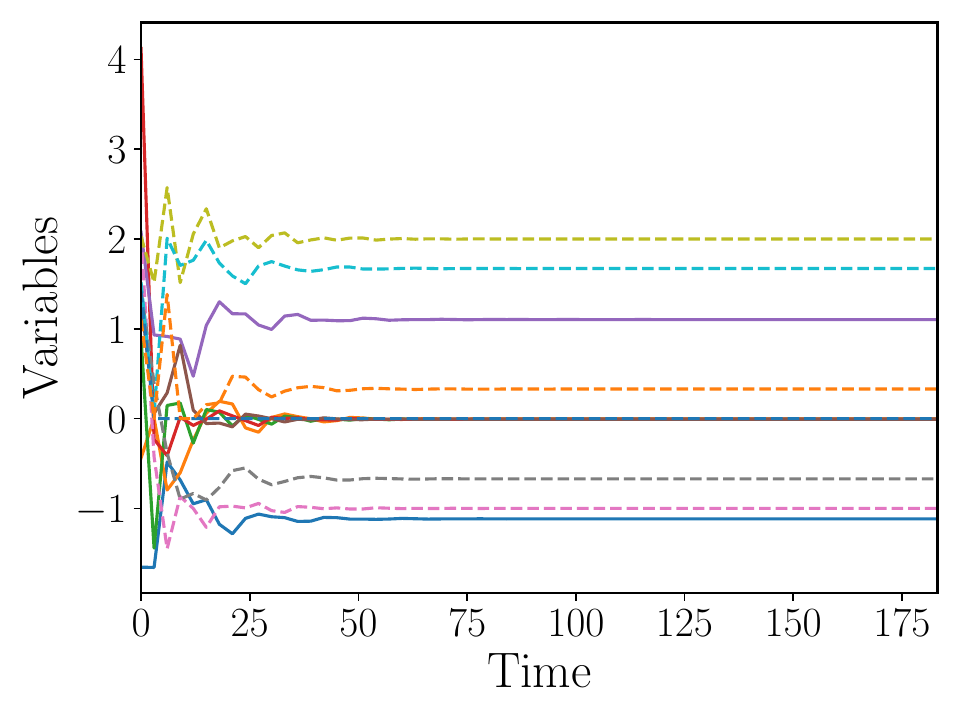}
        \caption{}
        \label{fig: 2a}
    \end{subfigure}%
    \hfill
    \begin{subfigure}[t]{0.25\textwidth}
        \centering
        \includegraphics[width=\linewidth]{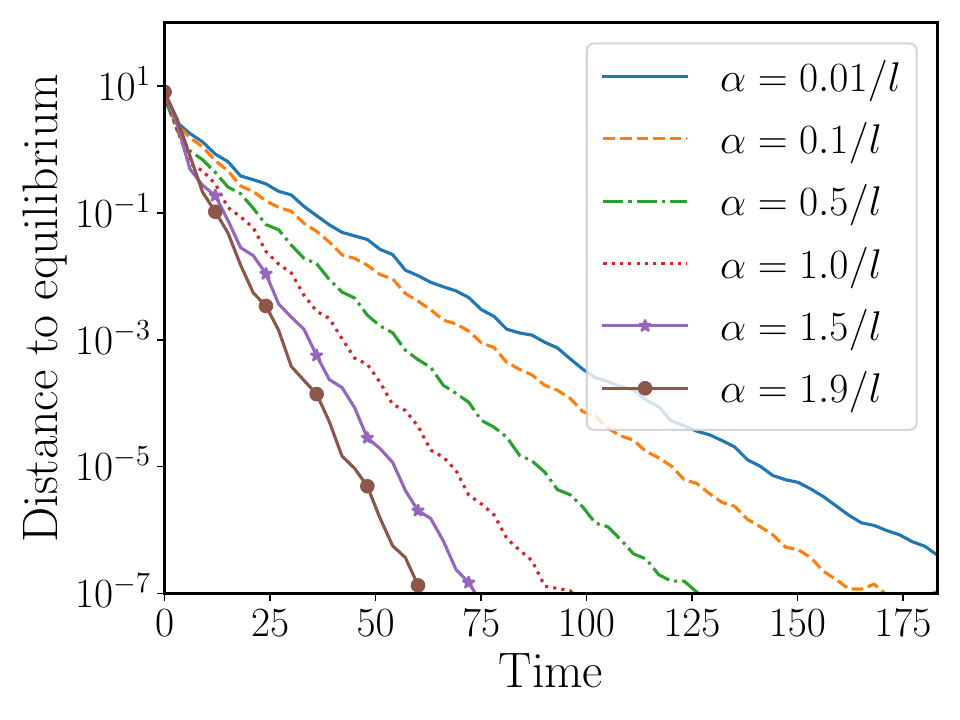}
        \caption{}
        \label{fig: 2b}
    \end{subfigure}

    \caption{(a) Network specification. The nodes are parameterized by injection (positive) or withdrawal (negative), and the edges are parameterized by capacity and unit cost. (b) Asymptotic convergence of trajectories for primal variable~$x$ with $\rho=0.05$. (c) Exponential convergence of trajectories for variables $u$ (solid lines) and $v$ (dashed lines) with $\alpha = \frac{0.01}{l}$. (d) Exponential decay of distance to equilibrium $\sqrt{\Vert u(t) - u_\star \Vert^2 + \Vert v(t) - v_\star \Vert^2  }$ with varying $\alpha$.}
    \label{fig:simulation-results}
\end{figure*}

\subsection{Network Flow Optimization}
We apply the augmented primal-dual dynamics \eqref{eq:reg-saddle-flow_proj} to a network flow optimization problem, which is commonly formulated as a linear program.
Consider a minimum-cost flow problem with a network specified in Fig.~\ref{fig: 1a}.
Let $\mathcal{V}$ and $\mathcal{E}$ denote the sets of the nodes and the directed edges, respectively.
$x:=(x_{ij},{(i,j)\in\mathcal{E}})\in\mathbb{R}^{|\mathcal{E}|}$ is the variable defined as the amount of flow on each edge $(i,j)$ from node~$i$ to node~$j$.
$c\in\mathbb{R}^{|\mathcal{E}|}$ and $b\in\mathbb{R}^{2|\mathcal{E}|}$ denote the unit cost of flow transmission and the capacity for all the edges, respectively. Note that $b$ is the concatenation of upper bounds (given in Fig.~\ref{fig: 1a}) and lower bounds (set to 0).
$d\in\mathbb{R}^{|\mathcal{V}|}$ denotes the nodal (net) injection.
Then given the incidence matrix $B\in\mathbb{R}^{|\mathcal{V}|\times|\mathcal{E}|}$ and the matrix $A:=\textrm{diag}(I_{|\mathcal{E}|},-I_{|\mathcal{E}|})$,
the minimum-cost flow problem is formulated as follows:
\begin{subequations}\label{eq:minimum-cost}
\begin{eqnarray}
         \min\limits_{x\in\mathbb{R}^{|\mathcal{E}|}} && c^Tx   \\
         \text{s.t.} && Bx-d=0 \ :  \ y^E\in\mathbb{R}^{|\mathcal{V}|}\\
         && Ax-b\leq 0 \ :  \ y^I\in\mathbb{R}^{2|\mathcal{E}|}_{\ge 0}
\end{eqnarray}
\end{subequations}
where the equality constraints enforce the flow conservation while the inequality constraints enforce the flow capacity.

Given the linear programming formulation~\eqref{eq:minimum-cost}, the augmented primal-dual dynamics~\eqref{eq:reg-saddle-flow_proj} can be applied to attain an optimal solution with asymptotic convergence guarantee. Note the projection is only defined on $y^I$.
Fig.~\ref{fig: 1b} shows three stages of the convergence process.
Initially, the trajectories exhibit large-amplitude oscillations in the first 10 seconds, as the system searches around for the region of optimal solutions.
The oscillations decay over the following 40 seconds. After approximately 50 seconds, the trajectories start to center around the equilibrium, and finally asymptotically converge to an optimal solution of the minimum flow cost problem \eqref{eq:minimum-cost}. 
\subsection{Lasso Regression}

We use a classical Lasso regression problem to test the performance of an algorithm that combines the preconditioning method in Section~\ref{ssec:exp_convergence_regularized_pdd} with proximal regularization. Consider the following problem
\begin{equation}\label{eq:Lasso}
    \min_{\hat x\in\mathbb{R}^n} \ \hat f(\hat x)+\lambda\Vert \hat x\Vert_1
\end{equation}
where $\lambda$ is a constant. 
$\hat f:\mathbb{R}^n \mapsto \mathbb{R}$ is convex with $l$-Lipschitz gradient, i.e.,~$lI \succeq \nabla^2 \hat f(\hat x) \succeq 0$ wherever~$\nabla^2 \hat f(\hat x)$ is defined.

To handle the non-smooth term~$\Vert \hat x\Vert_1$, we define $\bar x_i^+:=\max\{\hat x_i,0\}$ and~$\bar x_i^-=\max\{-\hat x_i,0\}$ for $\forall i=1,2,\dots,n$, such that
\begin{subequations}\label{eq: change of x}
    \begin{align}
        \hat x_i = \bar x_i^+ - \bar x_i^-, \\
        \vert \hat x_i\vert = \bar x_i^+ + \bar x_i^-,
    \end{align}
\end{subequations}
hold. 
Denote $\bar x^{+}:= (\bar x_i^+,i=1,2,\dots,n) \in\mathbb{R}^n$, $\bar x^{-}:= (\bar x_i^-,i=1,2,\dots,n) \in\mathbb{R}^n$, and $\bar x := [(\bar x^{+})^T, (\bar x^{-})^T]^T$.
The Lasso regression problem~\eqref{eq:Lasso} can be transformed equivalently into
\begin{subequations}\label{eq:modified_Lasso}
    \begin{eqnarray}
    \min_{\hat x\in\mathbb{R}^n,\ \bar x\in\mathbb{R}^{2n}} && \hat f(\hat x) + \lambda\mathbf{1}^T\bar x\\
    \mathrm{s.t.} \ \ \ \ && \hat x+C \bar x = 0 \ : \  \hat y \in \mathbb{R}^{n}\\
    && \bar x\geq 0 \ : \ \bar y \in \mathbb{R}^{2n}_{\ge 0}
    \end{eqnarray}
\end{subequations}
with~$C=[-I_n, I_n]$.
Given~$x := [\hat x^T,\bar x^T]^T$ and $y := [\hat y^T,\bar y^T]^T$,
we can define its Lagrangian as
\begin{equation}\label{eq:Lasso_Lagrangian}
    \begin{aligned}
        L(x,y) :&=  \hat f(\hat x) + \lambda\mathbf{1}^T \bar x + \hat y^T(\hat x+C\bar x)-\bar y ^T \bar x\\
        &= \underbrace{\hat f(\hat x) + \lambda\mathbf{1}^T\bar x}_{=:f(x)} + y^T\underbrace{\begin{bmatrix}
            I_n & C\\
            0 & -I_{2n}
        \end{bmatrix}}_A x \, .
    \end{aligned}
\end{equation}
Note that the matrix~$A$ is full row rank. Thus we bound the eigenvalues of $AA^T$ by $\sigma I\succeq   AA^T \succeq \kappa I$.

The Lagrangian \eqref{eq:Lasso_Lagrangian} is only convex-linear. 
We propose an algorithm that implements the primal-dual dynamics of a transformed function. The transformation involves two steps: first use the change of variables \eqref{eq:change of var} to attain a convex-strongly concave function and then apply a more general version of proximal regularization on the dual variable $y$ to attain the final function (recall our analysis of the proximal method in Section~\ref{ssec:proximal_reg} showing the ``shift" of strong convexity).
 
More specifically, with the change of variables \eqref{eq:change of var}, the Lagrangian \eqref{eq:Lasso_Lagrangian} is transformed into 
\begin{equation}\label{eq:regularied_Lasso_Lagrangian}
     L^C (u,y): = f(u-\alpha A^Ty) +  y^TAu-\alpha \|A^Ty\|^2 \, .
\end{equation}
By properly choosing the constant $\alpha>0$ such that $\alpha<\frac{2}{l}$, we can verify the convexity-strong concavity of ${L}^C(u,y)$ with
\begin{subequations}
    \begin{equation}
        \partial^2_{uu} {L}^C(u,y) = \nabla^2 f(u-\alpha A^Ty) \succeq  0 \, ,
    \end{equation}
    \begin{equation}
        \begin{aligned}
            \partial^2_{yy} {L}^C(u,y) &= \alpha^2   A\nabla^2 f(u-\alpha A^Ty)A^T - 2\alpha  A A^T  \\
            &=  A \left( \alpha^2 \nabla^2 f(u-\alpha A^Ty) - 2\alpha I  \right) A^T \\
            &\preceq  A ( l\alpha^2 -2\alpha)A^T \\
            & \preceq -\left(2{\alpha} - l\alpha^2\right) \kappa I \\
            & \prec 0 \, .
        \end{aligned}
    \end{equation}
\end{subequations}

In terms of $L^C(u,y)$, we apply the following proximal method on $y$
to attain the final transformed function:
\begin{small}
\begin{equation}\label{eq:Lasso_saddle}
\begin{aligned}L^P(u,v):=&\max\limits_{y\in\mathbb{R}^n\times\mathbb{R}^{2n}_{\geq 0}} \big\{ {L}^C(u,y)-\frac{\rho}{2}\|y-v\|^2 \big \} \\
        =&f(u-\alpha A^T\tilde{y})+\tilde{y}^TAu-\alpha\|A^T\tilde{y}\|^2-\frac{\rho}{2}\|\tilde{y}-{v}\|^2 \, ,
\end{aligned}
\end{equation}
\end{small}%
where~$\tilde{y}$ is a shorthand of $\tilde{y}(u,v)$, denoting the unique maximizer given~$(u,v)$.

We thus employ the saddle flow dynamics of~$L^P(u,v)$, i.e.,
\begin{subequations}\label{eq:saddle_flow_lasso}
\begin{align}
    \dot {u} & = -\nabla f\left(u-\alpha A^T\tilde{y}(u,v)\right)- A^T\tilde{y}(u,v) \, ,\\
    \dot {v} & =  \rho \tilde{y}(u,v)-\rho v \, .
\end{align}
\end{subequations}
to solve the Lasso regression problem \eqref{eq:Lasso}, since whenever \eqref{eq:saddle_flow_lasso} converges to an equilibrium point $(u_\star,v_\star)$, we can easily recover $(x_\star=u_\star - \alpha A^Tv_\star,y_\star=v_\star)$ as an optimal solution to \eqref{eq:Lasso}.
Note that the maximizer $\tilde{y}(u,v)$ can be efficiently computed using numerical methods since $\nabla f (\cdot)$ is strictly monotone.

In particular, we specify $\hat f(\hat x):=\frac{1}{2}\Vert A\hat{x}-b\|^2$, where $A\in\mathbb{R}^{m\times n}$ and $b\in\mathbb{R}^m$ could represent a collection of known input and output data of size $m$, respectively.

Fig.~\ref {fig: 2a} shows the trajectories of \eqref{eq:saddle_flow_lasso} that converge rapidly in approximately 40 seconds. Fig.~\ref{fig: 2b} further shows the exponentially decaying distance to the equilibrium as we implement the dynamics \eqref{eq:saddle_flow_lasso}. Moreover, the decaying rate increases in the tunable constant $\alpha$.

While the more general version of proximal regularization \eqref{eq:Lasso_saddle} goes beyond our analysis in Section~\ref{ssec:proximal_reg} due to the constrained feasible region $\mathbb{R}^n\times \mathbb{R}^{2n}_{\ge0}$ for $y$, the numerical results suggest that the desirable exponential stability of the saddle flow dynamics \eqref{eq:saddle_flow_lasso} may still hold, thus opening the path for future extensions to further relax our conditions.

\section{Conclusion}\label{sec:conclusion}
This paper focuses on the convergence behavior of saddle flow dynamics and provides a unified analysis that leads to two novel conditions for asymptotic and exponential convergence.
The first condition is an observable certificate which, if identified for a convex-concave function, guarantees the saddle flow dynamics to converge asymptotically. It generalizes some existing conditions and further inspires the design of a novel augmented algorithm that only requires minimal assumptions on convexity-concavity for asymptotic convergence and even works for bilinear functions. 
The second condition is the strong convexity-strong concavity of an objective function which suffices to guarantee saddle flows' global exponential stability and provides a convenient lower-bound estimate of the rate of convergence.
The insight also reveals how proximal regularization on a strongly convex-concave function leads to exponential convergence by ``shifting" some of the strong convexity to make the resulting function also strongly concave. 
Our analysis and results can be extended to a projected version of saddle flow dynamics, the trajectories of which are constrained within a closed convex polyhedron.
Such an extension allows the immediate application of our theory to develop and analyze primal-dual algorithms for constrained convex optimization. This leads to three novel algorithm designs and one novel bound on the convergence rate of proximal primal-dual algorithms, highlighting weaker convergence conditions and/or faster convergence rates.
Finally, we run extensive numerical experiments with our new algorithms applied to solve network flow optimization and Lasso regression. 
The simulations validate our theoretical development and also suggest potential extensions.
